\documentclass[11pt,twoside,english]{article}
\usepackage{mathptmx}
\usepackage{graphicx}

\usepackage[T1]{fontenc}
\usepackage[utf8]{inputenc}
\usepackage[a4paper]{geometry}
\usepackage{color}
\usepackage{babel}
\usepackage{mathrsfs}
\usepackage{amsmath}
\usepackage{amsthm}
\usepackage{amssymb}
\usepackage{esint}
\usepackage[unicode=true,
 bookmarks=true,bookmarksnumbered=true,bookmarksopen=false,
 breaklinks=false,pdfborder={0 0 1},backref=false,colorlinks=true]
 {hyperref}
\hypersetup{
 pdfauthor={G Xi}}

\makeatletter
\numberwithin{equation}{section}
\theoremstyle{plain}
\newtheorem{thm}{\protect\theoremname}[section]
\theoremstyle{plain}
\newtheorem{prop}[thm]{\protect\propositionname}
\theoremstyle{plain}
\newtheorem{lem}[thm]{\protect\lemmaname}
\theoremstyle{remark}
\newtheorem{rem}[thm]{\protect\remarkname}
\theoremstyle{plain}
\newtheorem{cor}[thm]{\protect\corollaryname}
\theoremstyle{remark}
\newtheorem*{rem*}{\protect\remarkname}

\usepackage{babel}
\allowdisplaybreaks[4]

\date{}

  \providecommand{\corollaryname}{Corollary}
  
  \providecommand{\lemmaname}{Lemma}
\providecommand{\theoremname}{Theorem}

\theoremstyle{plain}
\newtheorem{assumption}{Assumption}

\makeatother

\providecommand{\corollaryname}{Corollary}
\providecommand{\lemmaname}{Lemma}
\providecommand{\propositionname}{Proposition}
\providecommand{\remarkname}{Remark}
\providecommand{\theoremname}{Theorem}

\begin{document}
\global\long\def\divg{{\rm div}\,}%

\global\long\def\curl{{\rm curl}\,}%

\global\long\def\rt{\mathbb{R}^{3}}%

\global\long\def\rd{\mathbb{R}^{d}}%

\global\long\def\rtwo{\mathbb{R}^{2}}%

\global\long\def\e{\epsilon}%

\title{Incompressible viscous fluids in $\mathbb{R}^2$ and SPDEs on graphs, in presence of  fast advection and non smooth noise}
\author{Sandra Cerrai\thanks{Department of Mathematics, University of Maryland, College Park, MD
20742, USA. Emails: cerrai@umd.edu, gxi@umd.edu}\ \,\thanks{Partially supported by NSF grant DMS-1712934  {\em Analysis of Stochastic Partial Differential Equations with Multiple Scales}} \ and Guangyu Xi\footnotemark[1]}
\maketitle
\begin{abstract}
The asymptotic behavior of a class of stochastic reaction-diffusion-advection equations in the plane is studied. We show that as the divergence-free advection term becomes larger and larger, the solutions of such equations converge to the solution of a suitable stochastic PDE defined on the graph associated with the Hamiltonian. Firstly, we deal with the case that the stochastic perturbation is given by a singular spatially homogeneous Wiener process taking values in the space of Schwartz distributions. As in previous works, we assume here that the derivative of the period of the motion on the level sets of the Hamiltonian  does not vanish. Then, in the second part, without assuming this condition on the derivative of the period, we study a weaker type of convergence for the solutions of a suitable class of linear SPDEs. 

\end{abstract}

\section{Introduction}
In this paper we are interested in studying the limiting behavior of some particles that move together with an incompressible flow in $\mathbb{R}^2$, with stream function $H(x)$, under the assumption that the flow has a small viscosity and the particles are subject to a slow chemical reaction, which consists of a deterministic and a stochastic component. The density $v_\e(t,x)$ of the particles, at time $t\geq 0$ and position $x \in\,\mathbb{R}^2$,  satisfies the equation 
\begin{equation} \label{intro2}
\left\{
\begin{array}{l}
\displaystyle{\partial_t v_\e(t,x)=\frac \e 2\,\Delta v_\e(t,x) +\langle \nabla^\perp H(x),\nabla v_\e(t,x)\rangle+\e b(v_\e(t,x))+\sqrt{\e}\sigma(v_\e(t,x))\partial_t\mathcal{W}(t,x),}\\
\displaystyle{v_\e(0,x)=\varphi(x),\ \ \ \ x \in\,\mathbb{R}^2,}
\end{array}\right. 
\end{equation}
for some parameter $0<\epsilon \ll 1$. Throughout the paper, we assume that the Hamiltonian $H:\mathbb{R}^2\to\mathbb{R}$ is a generic function, having four continuous derivatives, with bounded second derivative, such that $H(x)\to \infty$, as $|x|\to\infty$. The nonlinearities $b, \sigma:\mathbb{R} \to \mathbb{R}$ are assumed to be Lipschitz continuous and $\mathcal{W}(t,x)$ is a spatially homogeneous Wiener process (see below for all details). 

It is immediate to check that, under these conditions, on any finite time interval $[0,T]$ the solutions $v_\e$ of equation \eqref{intro2} converge to the solution $v$ of the Liouville equation
\[\partial_t v(t,x)=\langle \nabla^\perp H(x),\nabla v(t,x)\rangle,\ \ \ \ v(0,x)=\varphi(x).\]
However, on  time intervals of order $\epsilon^{-1}$ the difference $v_\e-v$ is of order $1$, as $\epsilon\to 0$.  Actually, on such a time interval, the limiting behavior of $v_\e$ is described by a non-standard SPDE defined on the graph $\Gamma$ associated with the Hamiltonian $H$, which is obtained by identifying all points on the same connected component of each level set of $H$ (see Subsection 2.1 for the precise definition).
Such an asymptotic behavior of $v_\e$ has been studied in \cite{CerraiFreidlin2019}, under quite restrictive conditions on the regularity of the noise $\mathcal{W}(t)$ and under the assumption that the derivative of the period of the motion on the level sets of the Hamiltonian  $H$ does not vanish. In the present paper we want to understand what happens when  these conditions are not satisfied.

To this purpose, before proceeding with the description of the content of the paper, we would like to remark that the study of SPDEs on graphs is  still a quite new field of investigation and very few results are available in the existing literature. In addition to  the already mentioned paper \cite{CerraiFreidlin2019}, in \cite{CF} a class of SPDEs on graphs, obtained as limits of SPDEs in narrow tubes, is studied. Moreover in \cite{BMZ} first and then, more recently, in \cite{Fan2017}, suitable classes of SPDEs on graphs have been also considered. 

\medskip

With the time change $t\mapsto t/\epsilon$, for every fixed $\e>0$ the function $u_\e(t,x):=v_\e(t/\epsilon,x)$ satisfies the equation 
\begin{equation}
\label{intro1}
\left\{
\begin{array}{l}
\displaystyle{\partial_t u_\e(t,x)=L_\epsilon u_\epsilon(t,x)+ b(u_\e(t,x))+\sigma(u_\e(t,x))\partial_t\mathcal{W}(t,x),}\\
\displaystyle{u_\e(0,x)=\varphi(x),\ \ \ \ x \in\,\mathbb{R}^2,}
\end{array}\right. \end{equation}
where 
\[L_\epsilon \varphi(x)=\frac 12 \Delta \varphi(x)+\frac 1\epsilon \langle \nabla^\perp H(x),\nabla \varphi(x)\rangle.\]
The operator $L_\epsilon$ is the generator of the Markov semigroup $S_\e(t)$, $t\geq 0$, 
associated with the stochastic differential equation
\[dX_\e(t)=\frac 1\e\,\nabla^\perp H(X_\e(t))\,dt+dB(t),\]
where $B(t)$ is a Brownian motion in $\mathbb{R}^2$, defined on the stochastic basis $(\Omega, \mathcal{F}, \{\mathcal{F}_t\}_{t\geq 0}, \mathbf{P})$. More precisely, for every Borel and bounded  function $\varphi:\mathbb{R}^2\to \mathbb{R}$ and every $x \in\,\mathbb{R}^2$
\begin{equation}
\label{intro12}
S_\epsilon(t)\varphi(x)=\mathbf{E}_x \varphi(X_\epsilon(t)),\ \ \ \ \ t\geq 0.
\end{equation}
This means, in particular, that $u_\e$ is a  mild solution to equation \eqref{intro1} if 
\begin{equation}
\label{intro4}
u_\e(t)=S_\epsilon(t) \varphi(x)+\int_0^tS_\epsilon(t-s)B(u_\epsilon(s))\,ds+\int_0^t S_\epsilon(t-s)\Sigma(u_\epsilon(s))d\mathcal{W}(s),\end{equation}
where $B$ and $\Sigma$ are the composition/multiplication operators associated with $b$ and $\sigma$, respectively.

In \cite{CerraiFreidlin2019}, together with M. Freidlin, the first named author proved that  
for every $p\geq 1$ and $0<\tau<T$
\begin{equation}
\label{intro3}
\lim_{\epsilon \to 0} \mathbb{E}\sup_{t \in\,[\tau,T]}|u_\e(t)-\bar{u}(t)\circ \Pi|^p_{H_\gamma}=0,
\end{equation}
where $\bar{u}$ is the solution of an {\em averaged} SPDE defined on the graph $\Gamma$ and $H_\gamma$ is a suitable weighted space of square integrable functions on $\mathbb{R}^2$, with respect to  a finite measure $\gamma^\vee(x)\,dx$. 

Due to \eqref{intro4}, it is evident that the proof of \eqref{intro3} is based on the analysis of the limiting behavior of the semigroups $S_\epsilon(t)$, as $\e\downarrow 0$, for every $t \in\,[\tau,T]$. To this purpose, in \cite[Chapter 8]{FreidlinWentzell1998}, it is proved that if $\Pi$ is the projection of $\mathbb{R}^2$ onto $\Gamma$,  the slow process $Y_\e(\cdot):=\Pi(X_\e(\cdot))$, defined on the graph $\Gamma$, converges weakly in $C([0,T];\Gamma)$ to a continuous Markov process $\bar{Y}(\cdot)$ on $\Gamma$, whose generator $\bar{L}$ is explicitly given in terms of differential operators on each edge and suitable gluing conditions at the vertices.    
Hence, starting from such result, in \cite[Appendix A]{CerraiFreidlin2019} it has been shown that for every $\varphi \in\,C_b(\mathbb{R}^2)$ and for every $x \in\,\mathbb{R}^2$ and $0<\tau<T$
\begin{equation}
\label{intro5}
\lim_{\e\to 0} \sup_{t \in\,[\tau,T]}\left|S_\e(t)\varphi(x)-(\bar{S}(t) \varphi^\wedge)\circ \Pi (x)\right|=0,\end{equation}
where
\[\varphi^\wedge(z,k):=\frac 1{T_k(z)}\oint_{C_k(z)}\frac{\varphi(x)}{|\nabla H(x)|}\,dl_{z,k},\ \ \ (z,k) \in\,\Gamma,\]
$dl_{z,k}$ is the length element on $C_k(z)$, the $k$-th connected component  of $C(z):=\left\{x \in\,\mathbb{R}^2\,:\ H(x)=z\right\}$, and
\[T_k(z):=\oint_{C_k(z)}\frac 1{|\nabla H(x)|}\,dl_{z,k},\]
(for all details see Subsection \ref{ss2.1}).
Once identified the right weighted spaces $H_\gamma$ and proved limit \eqref{intro5}, it can be shown that for every $\varphi \in\,H_\gamma$
\begin{equation}
\label{intro6}
\lim_{\e\to 0}\sup_{t \in\,[\tau,T]}|S_\e(t)\varphi- (\bar{S}(t)\varphi^\wedge)\circ \Pi |_{H_\gamma}=
\lim_{\e\to 0}\sup_{t \in\,[\tau,T]}|(S_\e(t)\varphi)^\wedge- \bar{S}(t)\varphi^\wedge |_{\bar{H}_{{\gamma}}}=0.
\end{equation}
Here the choice of the weight $\gamma^\vee$ requires a non-trivial analysis, as it has to be admissible with respect to all semigroups $S_\e(t)$ and its projection $\gamma$ on $\Gamma$ has to be admissible with respect to $\bar{S}(t)$. Moreover, the space $H_\gamma=L^2(\rtwo,\gamma^\vee(x)dx)$ has to be properly projected into the space $\bar{H}_{{\gamma}}=L^2(\Gamma,\nu_\gamma)$, where $\nu_\gamma$ is the projection on $\Gamma$ of $\gamma^\vee(x)\,dx$ (see Subsection \ref{subsec2.3} and \cite{CerraiFreidlin2019} for all details). 

In \cite{CerraiFreidlin2019}, limit \eqref{intro6} is then used in \eqref{intro4}, to obtain limit \eqref{intro3}. Taking the limit, as $\e\to  0$, in the first two terms on the right-hand side in \eqref{intro4} is an immediate consequence of \eqref{intro6} and the Lipschitz-continuity of the non-linearity $b$. On the other hand, taking  the limit in the last term, the stochastic integral, requires some extra effort and, most importantly, requires the spatially homogeneous Wiener process $\mathcal{W}$ to be smooth. In particular, in \cite{CerraiFreidlin2019} it is assumed that its spectral measure is finite, so that $\mathcal{W}(t,\cdot)$ takes values in the functional space $H_\gamma$. Moreover, the proof of \eqref{intro3} requires  the condition
\begin{equation}
\label{intro10}
\frac{dT_{k}(z)}{dz}\neq0,\qquad(z,k)\in\Gamma.
\end{equation}
 This assumption is needed for the proof of \eqref{intro5}. Actually, \eqref{intro5} and hence \eqref{intro3} still stand if \eqref{intro10} is true except for a finite number of points on the graph $\Gamma$. But it is easy to check that important examples such as $H(x)=\vert x\vert ^2$, for which the graph is $[0,\infty)$ and the period $T(z)\equiv \pi$, are still excluded by such an assumption. 

In the first part of the present paper, we are interested in understanding if limit \eqref{intro3} is still valid, under the minimal assumptions on the spectral measure $\mu$ that assure the well posedness of equation \eqref{intro1} in the space $H_\gamma$ (see \cite{PeszatZabczyk1997} and Assumption \ref{Assumption 3}). In section \ref{section3}, assuming that the spectral measure to the singular spatially homogeneous Wiener process $\mathcal{W}(t)$ in $\rtwo$ has a density function $m$ in $L^p(\rtwo)$ for some $p\in(1,\infty)$ and \eqref{intro10} holds, we prove that \eqref{intro3} is still valid (see Theorem \ref{thm: Convergence of SPDE}). Actually, with little modification to our proof, we can further extend Theorem \ref{thm: Convergence of SPDE} to singular spatially homogeneous Wiener processes with spectral measure 
\[\mu=\mu_1+\mu_2,\]
where $\mu_1$ is a finite measure and $\mu_2$ has density function $m\in L^p(\rtwo)$ for some $p\in(1,\infty)$. This combines the results of \cite{CerraiFreidlin2019} and section \ref{section3}, and covers a large class of spatially homogeneous Wiener processes (for specific examples of the processes, we refer to \cite{PeszatZabczyk1997}).

To understand the convergence of the solutions to the SPDEs under singular spatially homogeneous Wiener process, in section \ref{section3} we first study the properties of the semigroups $S_\e(t)$ and their limit $\bar{S}(t)$. For this purpose, we introduce the kernel $G_\e(t,x,y)$ of the semigroup $S_\e(t)$, and we prove that
\begin{equation}
\label{intro7}
\sup_{\e>0}\,G_{\epsilon}(t,x,y)\leq\frac{C}{t}\exp\left(-\frac{(\sqrt{H(y)+1}-\sqrt{H(x)+1})^{2}}{4Ct}\right),\end{equation}
for any $(t,x,y)\in(0,T]\times\rtwo\times\rtwo$. Notice that due to \eqref{intro5} we have that the semigroup $\bar{S}(t)^\vee$, defined by
\[\bar{S}(t)^\vee\varphi(x):=(\bar{S}(t) \varphi^\wedge)\circ \Pi (x),\ \ \ \ x \in\,\rtwo,\ \ \ t\geq 0,\]
admits a kernel $\bar{G}(t,x,y)$, which satisfies estimate \eqref{intro7} as well.

Now, given a spatially homogeneous Wiener process $\mathcal{W}(t)$ in $\rtwo$ with spectral measure $m\in L^p(\rtwo)$ for some $p\in(1,\infty)$, we define  $\bar{\mathcal{W}}(t)$ to be the projection of $\mathcal{W}(t)$ on $\Gamma$. We denote by $\mathscr{S}_{q}'$ and $\bar{\mathscr{S}}_{q}'$ the reproducing kernels of the Wiener processes $\mathcal{W}(t)$ and $\bar{\mathcal{W}}(t)$, respectively. Using \eqref{intro7}, we prove that for every $T>0$ there exists a constant $C_T>0$ such that
\[
\sum_{j=1}^{\infty}\left|S_{\epsilon}(t)\left(\psi e_{j}\right)\right|_{H_{\gamma}}^{2}\leq C_T \Vert m\Vert_{L^p}\,t^{-(p-1)/p}\vert\psi\vert_{H_{\gamma}}^{2},\ \ \ \ \ t \in\,(0,T],
\]
and 
\[
\sum_{j=1}^{\infty}\left|\bar{S}(t)^{\vee}\left(\psi e_{j}\right)\right|_{H_{\gamma}}^{2}\leq C_T \Vert m\Vert_{L^p}\,t^{-(p-1)/p}\vert\psi\vert_{\bar{H}_{\gamma}}^{2},\ \ \ \ \ t \in\,(0,T],
\]
where $\{e_j\}_{j \in\,\mathbb{N}}$ is the orthonormal basis of $\mathscr{S}_{q}'$. This, in particular, allows us to   prove the well-posedness of the SPDEs \eqref{intro1} in $H_\gamma$. Next, for the convergence of the solutions $u_\e$ to $\bar{u}$, we need a stronger type of  convergence for the semigroups. In fact,  by using a suitable decomposition of the density function $m$ of the spectral measure, we prove that for any $\psi \in\,H_\gamma$
\begin{equation}
\label{intro11}
\lim_{\epsilon\rightarrow0}\sup_{t \in\,[\tau,T]}\,\sum_{j=1}^{\infty}\left|(S_{\epsilon}(t)-\bar{S}(t)^{\vee})\left(\psi e_{j}\right)\right|_{H_{\gamma}}^{2}=0.
\end{equation}
Thanks to \eqref{intro11}, we can then handle the convergence of the stochastic integral in \eqref{intro4} and prove \eqref{intro3}.

\medskip

In the second part of this paper we try to understand what happens when condition \eqref{intro10} does not hold. We recall that such condition
 is needed in both \cite{CerraiFreidlin2019} and section 3. This assumption is necessary for proving \eqref{intro5} and hence \eqref{intro3}, i.e. the convergence of  $S_\e(t)\varphi$ to $\bar{S}(t)^\vee \varphi$  for any fixed time $t>0$ and $\varphi \in\,H_\gamma$. Thanks to \eqref{intro12}, it is easy to see that \eqref{intro3} is equivalent to 
\begin{equation}
\label{intro15}
\lim_{\epsilon\rightarrow0}\sup_{t\in[\tau,T]}\left|\mathbf{E}_{x}u(X_{\epsilon}(t))-\bar{\mathbf{E}}_{\Pi(x)}u^{\wedge}(\bar{Y}(t))\right|=0.
\end{equation}
Without assuming \eqref{intro10}, clearly \eqref{intro15} is no longer true, as can be shown in the case   $H(x)=\vert x\vert ^2$. Nevertheless,  in section \ref{section4}, (see Theorem \ref{teo1}) we prove that a weaker type of convergence holds. Namely, 
\begin{equation}
\label{intro16}
\lim_{\epsilon\rightarrow0}\sup_{x\in K}\left|\int_{\tau}^{T}\left[\mathbf{E}_{x}u(X_{\epsilon}(t))-\bar{\mathbf{E}}_{\Pi(x)}u^{\wedge}(\bar{Y}(t))\right]\theta(t)dt\right|=0,
\end{equation}
for any compact set $K\subset \rtwo$, $u\in C_{b}(\rtwo)$ and $\theta\in C_{b}([\tau,T])$. 

Using \eqref{intro16}, we further study the convergence of the SPDEs. Since limit \eqref{intro16} is not preserved by the nonlinearities $b$ and $\sigma$, we restrict our consideration to the   linear case
\[\left\{\begin{array}{l}
\partial_{t}u_{\epsilon}(t,x)  =\frac{1}{2}\Delta u_{\epsilon}(t,x)+\frac{1}{\epsilon}\langle\nabla^{\perp}H(x),\nabla u_{\epsilon}(t,x)\rangle+\partial_{t}\mathcal{W}(t,x),\\
u_{\epsilon}(0,x)  =\varphi(x),\qquad x\in\mathbb{R}^{2}.
\end{array}\right.\]
In this case, we show   that
\[\lim_{\epsilon\rightarrow0}\mathbb{E}\left|\int_{0}^{T}\left[u_{\epsilon}(t)-\bar{u}(t)^{\vee}\right]\theta(t)dt\right|_{H_{\gamma}}^{q}=0,\]
(see  Theorem \ref{thmLinear}).

\medskip

The structure of the paper is as follows. In Section 2, we introduce the necessary notations and preliminaries from previous works. In Section 3 we prove our first main result stated in Theorem \ref{thm: Convergence of SPDE}. Under the assumption that the density of the spectral measure is in $L^p(\rtwo)$, for some 
$p\in(1,\infty)$, we first study the properties of the semigroups and the well posedness of the SPDEs. Then we prove Theorem \ref{thm: Convergence of SPDE}. In section 4, we prove that if condition \eqref{intro10} is not satisfied, then a weaker type of convergence of the semigroups $S_\e(t)$ holds. Next, we prove that this implies a weaker type of convergence for the solutions of a class of linear SPDEs.

\section{Notations and preliminaries}
In this section, we introduce the notations that will be used in later sections. For the completeness of the paper, we also briefly recall the results in previous works, which will be used in our work here. 

To study the convergence of the SPDEs, we first need to understand the convergence of the semigroups $S_\epsilon(t)$. In section \ref{F-WAveragingResult}, we briefly recall the Freidlin-Wentzell averaging results in \cite{FreidlinWentzell1998}. Then in section \ref{subsec2.3}, we recall some properties of the weighted spaces $H_\gamma$ and $\bar{H}_\gamma$ proved in \cite{CerraiFreidlin2019}, which will be used when studying the solutions to the SPDEs that fall in the weighted spaces. Finally, the random forcing  $\mathcal{W}(t,x)$ in the SPDEs are assumed to be {\em spatially homogeneous Wiener processes} with positive-symmetric spectral measure $\mu$ on $\rtwo$. We recall the main definitions and properties of the spatially homogeneous Wiener process in section \ref{SHWPs} following \cite{PeszatZabczyk1997}.

\subsection{The Hamiltonian and the associated graph }
\label{ss2.1}
Throughout this paper, we consider the Hamiltonian system
\begin{equation}
dx(t)=\nabla^{\perp}H(x(t)),\ \ \ \ x \in\,\rtwo,\label{eq: Hamilton's equation}
\end{equation}
where \[
\nabla^{\perp}H(x)=\left(\frac{\partial H(x)}{\partial x_{2}},-\frac{\partial H(x)}{\partial x_{1}}\right),\ \ \ \ x \in\,\rtwo.
\]
We shall assume that the Hamiltonian  $H$ satisfies the following conditions.
\begin{assumption}\label{Assumption 1} The Hamiltonian $H:\rtwo \rightarrow \mathbb{R}$ satisfies that
\begin{enumerate}
\item $H$ is four times continuously differentiable, with bounded second
derivatives. It has only a finite number of  critical points $x_{1},\cdots,x_{n}$,
and they are all non-degenerate. Moreover, 
\[H(x_{i})\neq H(x_{j}),
\ \ \ \mbox{if } \  i\neq j;\]
\item There exists $a>0$ such that for all $x\in\rtwo$ with $\vert x\vert$
large enough, we have
\[H(x)\geq a\vert x\vert^{2}, \qquad  \vert\nabla H(x)\vert\geq a \vert x\vert,\qquad  \Delta H(x)\geq a;\]
\item We have $\min_{x\in\rtwo}H(x)=0$.
\end{enumerate}
\end{assumption}

For any $z\geq0$, we denote by $C(z)$  the $z$-level set  of the Hamiltonian $H$
\[
C(z):=\left\{ x\in\rtwo:H(x)=z\right\} =\bigcup_{k=1}^{N(z)}C_{k}(z),
\]
where $C_k(z)$, $k=1,\ldots,N(z)$, are all the connected components of $C(z)$.
If we denote by $k(x)$ the number of the connected component of $C(H(x))$ containing $x$, then 
\[x(0)=x\Longrightarrow x(t) \in\,C_{k(x)}(H(x)),\ \ \ \ \ t\geq 0.\]
If $z$ is not a critical value, each $C_{k}(z)$ is a one periodic
trajectory of the Hamiltonian system (\ref{eq: Hamilton's equation}), and 
\begin{equation}
\label{cx5}
T_{k}(z):=\oint_{C_{k}(z)}\frac{1}{\vert\nabla H(x)\vert}dl_{z,k}
\end{equation}
 is the period of the motion along the level set $C_{k}(z)$ (here $dl_{z,k}$ is the length element
on $C_{k}(z)$).
Moreover, the probability measure 
\[
d\mu_{z,k}:=\frac{1}{T_{k}(z)}\frac{1}{\vert\nabla H(x)\vert}dl_{z,k}
\]
is invariant for the Hamiltonian equation (\ref{eq: Hamilton's equation})
on the level set $C_{k}(z)$

Now, by identifying the points on the same connected components $C_{k}(z)$,
we obtain a graph $\Gamma$. We denote by $\Pi:\rtwo\rightarrow\Gamma$
the identification map.
The graph $\Gamma$ consists of edges $I_{0},\cdots,I_{n}$ and vertices
$O_{0},\cdots,O_{m}$. The vertices are of two types,
external and internal vertices. External vertices correspond to local extrema of $H$, while internal vertices correspond to saddle points of $H$. Among external vertices, we denote by
$O_{0}$ the vertex corresponding to the point at infinity and by $I_{0}$
the only unbounded edge connected to $O_{0}$ (see Figure 1).
\begin{figure}
\centering
\includegraphics[height=8cm, width=12cm, bb=38 6 459 357]{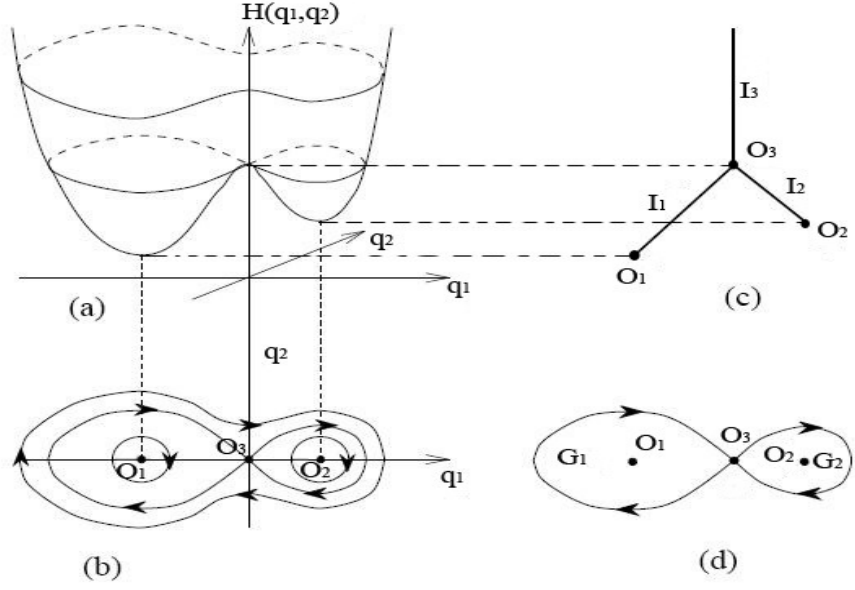}
\caption{The Hamiltonian, the level sets, the projection and the graph}
\end{figure}

On graph $\Gamma$,
a distance can be introduced as follows. If two points $y_{1}$ and
$y_{2}$ on the graph are on the same edge $I_{k}$, i.e. $y_{1}=(z_{1},k)$
and $y_{2}=(z_{2},k)$, then $d(y_{1},y_{2})=\vert z_{1}-z_{2}\vert$.
If $y_{1}$ and $y_{2}$ are on different edges, then 
\[
d(y_{1},y_{2})=\min\left\{ d(y_{1},O_{i_{1}})+d(O_{i_{1}},O_{i_{2}})+\cdots+d(O_{i_{j}},y_{2})\right\} ,
\]
where the minimum is taken over all possible paths from $y_{1}$ to
$y_{2}$, through every possible sequences of vertices $O_{i_{1}},\cdots,O_{i_{j}}$,
connecting $y_{1}$ and $y_{2}$.  Corresponding to each edge $I_{k}$, there
is an open set 
\[
G_{k}=\left\{ x\in\rtwo:\Pi(x)\in\mathring{I}_{k}\right\} .
\]
For $0\leq z_{1}<z_{2}$, we can define
\[
G(z_{1},z_{2})=\left\{ x\in\rtwo:z_{1}<H(x)<z_{2}\right\} ,
\]
and 
\[
G_{k}(z_{1},z_{2})=\left\{ x\in G_{k}:z_{1}<H(x)<z_{2}\right\} .
\]
Given $\delta>0$, we set 
\[
G(\pm\delta)=\bigcup_{i=1}^{m}G^{i}(\pm\delta)=\bigcup_{i=1}^{m}\left\{ x\in\rtwo:H(O_{i})-\delta<H(x)<H(O_{i})+\delta\right\} .
\]
For each vertex $O_{i}$, we denote
\[
D^{i}=\left\{ x\in\rtwo:\Pi(x)=O_{i}\right\} .
\]
In addition,  given any edge $I_{k}$ connected to the vertex $O_{i}$,
we denote 
\[
D_{k}^{i}=D^{i}\cap\bar{G}_{k}.\] If an edge
$I_{k}$ is connected to a vertex $O_{i}$, we write $I_{k}\sim O_{i}$.
For each $\delta>0$ and $I_{k}\sim O_{i}$, we set 
\[
D(\pm\delta)=\bigcup_{i=1}^{m}\bigcup_{k:I_{k}\sim O_{i}}D_{k}^{i}(\pm\delta)=\bigcup_{i=1}^{m}\bigcup_{k:I_{k}\sim O_{i}}\left\{ x\in G_{k}:d(\Pi(x),O_{i})=\delta\right\} .
\]
For further details, we refer to \cite[Chapter 8]{FreidlinWentzell1998} and \cite{CerraiFreidlin2019}.

\subsection{The Freidlin-Wentzell averaging result}\label{F-WAveragingResult}
With a change of time in \eqref{eq: Hamilton's equation}, for every $\e>0$, the function $x_\epsilon(t):=x(t/\epsilon)$ satisfies the equation
\begin{equation}
dx_{\epsilon}(t)=\frac{1}{\epsilon}\nabla^{\perp}H(x_{\epsilon}(t)).\label{eq: deterministic fast motion}
\end{equation}
Now, suppose $B_{t}$
is a standard Brownian motion on $\rtwo$. For every $\e>0$, we denote by $X_\e(t)$ the solution of the stochastic differential equation
\begin{equation}
dX_{\epsilon}(t)=\frac{1}{\epsilon}\nabla^{\perp}H(X_{\epsilon}(t))dt+dB(t).\label{eq: stocahstic fast motion}
\end{equation}
The second order differential operator associated with   (\ref{eq: stocahstic fast motion})
is
\[
L_{\epsilon}u(x)=\frac{1}{2}\Delta u(x)+\frac{1}{\epsilon}\langle\nabla^{\perp}H(x),\nabla u(x)\rangle.
\]
In what follows, we shall denote by $S_\e(t)$ the corresponding Markov transition semigroup. We recall that, for every Borel bounded $u:\mathbb{R}^2\to \mathbb{R}$, there is
\[S_\e(t)u(x)=\mathbf{E}_x u(X_\e(t)), \qquad \mbox{for } \ x \in\,\rtwo,\  t\geq 0.\]

Now, for every $x \in\,\mathbb{R}^2$, we consider the process $\Pi(X_\e(t))$, $t\geq 0 $, defined on the graph $\Gamma$, with $X_\e(0)=x$. In \cite[Chapter 8]{FreidlinWentzell1998}, is studied the limiting behavior, as $\e\downarrow 0$, of the process $\Pi(X_\e)$ in the space  $C([0,T];\Gamma)$, for any fixed $T>0$ and $x \in\,\mathbb{R}^2$. Namely, in \cite[Theorem 8.2.2]{FreidlinWentzell1998} it has been proved  that if the Hamiltonian $H$ satisfies Assumption \ref{Assumption 1}, the process $\Pi(X_\e)$, which describes the slow motion of $X_\e$, converges, in the sense of weak convergence of distributions in the space of continuous $\Gamma$-valued functions, to a diffusion process $\bar{Y}$ on $\Gamma$.

The process $\bar{Y}$ has been described in \cite[Theorem 8.2.1]{FreidlinWentzell1998}  in terms of its generator $\bar{L}$. The operator $(\bar{L}, D(\bar{L}))$ is a non-standard operator, which is given by suitable differential operators $\bar{L}_k$ within each edge $I_k$ of the graph and by certain gluing conditions at the interior vertices $O_i$ of the graph.
Moreover, it is degenerate at the vertices of the graph. Nevertheless, in  \cite[Theorem 8.2.1]{FreidlinWentzell1998} it is shown that it is the generator of a Markov process $\bar{Y}$ on the graph $\Gamma$. In what follows, we shall denote by $\bar{S}(t)$ the semigroup associated with $\bar{Y}$, defined by
\[\bar{S}(t) f(z,k)=\mathbf{E}_{(z,k)}f(\bar{Y}(t)),\]
for every bounded Borel function $f:\Gamma\to \mathbb{R}$.

\subsection{The weighted spaces $H_\gamma$ and $\bar{H}_\gamma$}
\label{subsec2.3}
For any $u:\rtwo\to \mathbb{R}$ and $0\leq z_{1}<z_{2}$, we have
\[
\int_{G(z_{1},z_{2})}u(x)dx=\sum_{k=0}^{n}\int_{I_{k, z_1, z_2}}\oint_{C_{k}(z)}\frac{u(x)}{\vert\nabla H(x)\vert}dl_{z,k}dz,
\]
where
\[I_{k, z_1, z_2}:=\left\{(z,k) \in\,I_k\,:\ z \in\,[z_1,z_2]\right\}.\]
In particular, it holds that
\[
\int_{\rtwo}u(x)dx=\sum_{k=0}^{n}\int_{I_{k}}\oint_{C_{k}(z)}\frac{u(x)}{\vert\nabla H(x)\vert}dl_{z,k}dz.
\]
In what follows, for every $u:\rtwo \to \mathbb{R}$, we shall define
\[
u^{\wedge}(z,k)=\frac{1}{T_{k}(z)}\oint_{C_{k}(z)}\frac{u(x)}{\vert\nabla H(x)\vert}dl_{z,k}=\oint_{C_{k}(z)}u(x)d\mu_{z,k},\ \ \ \ (z,k) \in\,\Gamma.
\]
Moreover, for every $f:\Gamma\to \mathbb{R}$, we shall define
\[
f^{\vee}(x)=f(\Pi(x)),\ \ \ \ \ x \in\rtwo.
\]
With these notations, given a positive continuous function $\gamma$ on the graph $\Gamma$,
if we assume that 
\[
\sum_{k=0}^{n}\int_{I_{k}}\gamma(z,k)T_{k}(z)dz<\infty,
\]
then $\gamma^{\vee}\in L^{1}(\rtwo)\cap C_{b}(\rtwo)$. For any such
function $\gamma$, we define 
\[
H_{\gamma}=\left\{ u:\rtwo\rightarrow\mathbb{R}:\vert u\vert_{H_{\gamma}}^{2}=\int_{\rtwo}\vert u(x)\vert^{2}\gamma^{\vee}(x)dx<\infty\right\} ,
\]
and 
\[
\bar{H}_{\gamma}=\left\{ f:\Gamma\rightarrow\mathbb{R}:\vert f\vert_{\bar{H}_{\gamma}}^{2}=\sum_{k=0}^{n}\int_{I_{k}}\vert f(z,k)\vert^{2}\gamma(z,k)T_{k}(z)dz<\infty\right\} .
\]
We recall the following results proved in \cite{CerraiFreidlin2019}.
\begin{prop}\label{prop:weightedSpaces}
For every $u\in H_{\gamma}$, we have $u^{\wedge}\in\bar{H}_{\gamma}$
and for every $f\in\bar{H}_{\gamma}$, we have $f^{\vee}\in H_{\gamma}$. Moreover,
\begin{equation}
\label{1}
\vert u^{\wedge}\vert_{\bar{H}_{\gamma}}\leq\vert u\vert_{H_{\gamma}}, \qquad  \vert f^{\vee}\vert_{H_{\gamma}}=\vert f\vert_{\bar{H}_{\gamma}}.\end{equation}
Finally, if $u\in H_{\gamma}$ and $f\in\bar{H}_{\gamma}$, then
\begin{equation}
\label{2}
\langle f,u^{\wedge}\rangle_{\bar{H}_{\gamma}}=\langle f^{\vee},u\rangle_{H_{\gamma}}, \qquad (f^{\vee}u)^{\wedge}=fu^{\wedge}.\end{equation}
\end{prop}

Now, for every linear operator $Q\in\mathcal{L}(H_{\gamma})$ and $A\in\mathcal{L}(\bar{H}_{\gamma})$, we define
\[Q^{\wedge}f:=(Qf^{\vee})^{\wedge}, \qquad A^{\vee}u:=(Au^{\wedge})^{\vee}\]
for $f\in\bar{H}_{\gamma}$ and $u\in H_{\gamma}$.
Moreover, It can be proved that
\begin{equation}
\Vert Q^{\wedge}\Vert_{\mathcal{L}(\bar{H}_{\gamma})}\leq\Vert Q\Vert_{\mathcal{L}(H_{\gamma})}, \qquad \Vert A^{\vee}\Vert_{\mathcal{L}(H_{\gamma})}\leq\Vert A\Vert_{\mathcal{L}(\bar{H}_{\gamma})}.
\end{equation}

\subsection{Spatially homogeneous Wiener processes }
\label{SHWPs}
Let $(\Omega,\mathcal{F},\mathbb{P})$ be a complete
probability space with filtration $(\mathcal{F}_{t})_{t\geq0}$ and let  $\mathscr{S}$ be the Schwartz space with its dual space $\mathscr{S}'$ (the space of Schwartz or tempered distributions). 
 We say that $\mathcal{W}(t)$ is a {\em Wiener process}, defined on $\Omega$
and taking values in $\mathscr{S}'$, if for each $\psi\in\mathscr{S}$, the mapping $t\rightarrow\langle\mathcal{W}(t),\psi\rangle$
defines a Wiener process. In particular, there exists a bilinear continuous symmetric
positive-definite form $Q:\mathscr{S}\times\mathscr{S}\rightarrow\mathbb{R}$
such that 
\[
\mathbb{E}\langle\mathcal{W}(t),\psi\rangle\langle\mathcal{W}(t),\varphi\rangle=t\wedge s\ Q(\psi,\varphi).
\]
In addition, we say that the Wiener process $\mathcal{W}(t)$ is {\em spatially
homogeneous} if the law of $\mathcal{W}(t)$ is invariant under all
translations $\tau_{h}(f)(x):=f(x+h)$ with $h\in\rtwo$. This implies that the bilinear form $Q$ must be of the form
\[
Q(\psi,\varphi)=\langle\Lambda,\psi\ast\varphi_{(s)}\rangle,
\]
where $\Lambda\in\mathscr{S}'$ is the Fourier transform of a
positive-symmetric tempered measure $\mu$ on $\mathbb{R}^{d}$, and
$\varphi_{(s)}(x)=\varphi(-x)$.  $\mu$ is called the {\em spectral measure } of $\mathcal{W}(t)$.

In what follows, we shall introduce in $\mathcal{S}$ the norm $q([\psi])=\sqrt{Q(\psi,\psi)}$  and we shall denote by $\mathscr{S}_{q}$
the completion of the set $\mathscr{S}/KerQ$ under the norm $q$.
The space $\mathscr{S}_{q}'$ is dual to $\mathscr{S}_{q}$ and can
be represented by 
\[
\mathscr{S}_{q}'=\{\xi\in\mathscr{S}':\exists C>0\  \mbox{with}\  \vert\langle\xi,\psi\rangle\vert\leq Cq([\psi]),\mbox{ for all }\psi\in\mathscr{S}\}.
\]
It turns out  that $\mathscr{S}_{q}'$ is the reproducing kernel
of the Wiener process $\mathcal{W}(t)$. 

Now, suppose $L_{(s)}^{2}(\mathbb{R}^{2},d\mu)$
is the space of all functions $u \in\,L^{2}(\mathbb{R}^2,d\mu)$  such that $u_{(s)}=u$.
As shown in \cite[Proposition 1.2]{PeszatZabczyk1997}, a distribution $\xi$ belongs to $\mathscr{S}_{q}'$
iff there exists a $u\in L_{(s)}^{2}(\mathbb{R}^2,d\mu)$ such that
$\xi=\widehat{u\mu}$. Moreover, for every $u, v \in\,L_{(s)}^{2}(\mathbb{R}^2,d\mu)$
 \begin{equation}
 \label{prop:spatially homogeneous WP property}
 \langle\widehat{u\mu},\widehat{v\mu}\rangle_{\mathscr{S}_{q}'}=\langle u,v\rangle_{L^{2}(\mathbb{R}^2,d\mu)}.\end{equation}

In what follows, we shall assume the following.

\begin{assumption}\label{Assumption 3}

The spectral measure $\mu$ of the spatially homogeneous
Wiener process has density function $m\in L^{p}(\rtwo)$, with $p\in(1,\infty)$.

\end{assumption}

 In particular, for any $u\in L_{(s)}^{2}(\mathbb{R}^2,d\mu)$
we have that 
\[\Vert um\Vert_{2p/(p+1)}\leq\Vert u\Vert_{L^{2}(\mathbb{R}^2,d\mu)}\Vert m\Vert_{p}^{1/2}.\]
Notice that $1\leq2p/(p+1)\leq2$, then by the Hausdorff-Young inequality we
have that 
\[
\Vert\widehat{um}\Vert_{2p/(p-1)}\leq C_{p}\Vert u\Vert_{L^{2}(\mathbb{R}^2,d\mu)}\Vert m\Vert_{p}^{1/2}.
\]
This implies that $\mathscr{S}_{q}'\subset L^{2p/(p-1)}(\mathbb{R}^2)$. Let  $\{u_{j}\}_{j\in\mathbb{N}}$
be an orthonormal basis of $L_{(s)}^{2}(\rtwo,\mu)$. According
to  \eqref{prop:spatially homogeneous WP property}, the functions $e_{j}:=\widehat{u_{j}m}$
define an orthonormal complete system in $\mathscr{S}_{q}'$, and the
spatially homogeneous Wiener processes can be represented as 
\[
\mathcal{W}(t,x)=\sum_{j=1}^{\infty}\widehat{u_{j}m}(x)\beta_{j}(t),
\]
where $\{\beta_{j}\}_{j\in\mathbb{N}}$ is a sequence of independent
Brownian motions. 
In particular, the corresponding Wiener process on the graph can
be written as
\begin{equation}
\label{cx6}
\bar{\mathcal{W}}(t,z,k)=\sum_{j=1}^{\infty}(\widehat{u_{j}m})^{\wedge}(z,k)\beta_{j}(t).
\end{equation}
We shall denote the reproducing kernel of $\bar{\mathcal{W}}$ by $\bar{\mathscr{S}_{q}'}$.

\section{The SPDE on $\mathbb{R}^2$ and the SPDE on the graph $\Gamma$}
\label{section3}

In this section, we consider the SPDE in $\mathbb{R}^2$
\begin{equation}
\label{SPDE}
\left\{\begin{array}{l}
\partial_{t}u_{\epsilon}(t,x)=L_\epsilon u_\epsilon(t,x)+b(u_{\epsilon}(t,x))+\sigma(u_{\epsilon}(t,x))\partial_{t}\mathcal{W}(t,x),\\
u_{\epsilon}(0,x)=\varphi(x),
\end{array}
\right.\end{equation}
 where we recall
\[L_\e u(x)=\frac 12 \Delta u(x)+\frac 1\epsilon \langle \bar{\nabla}H(x),\nabla u(x)\rangle,\ \ \ \ \ x \in\,\rtwo.\]

In what follows, we shall assume 
the following condition on the coefficients $b$ and $\sigma$.
\begin{assumption}\label{Assumption 4}
The nonlinearities  $b, \sigma:\mathbb{R} \to \mathbb{R}$ are  Lipschitz
continuous.

\end{assumption}
For every $u \in\,H_\gamma$ and $v \in\,\mathcal{S}^\prime_q$, we shall denote by
\[B(u)(x)=b(u(x)),\ \ \ \ [\Sigma(u)v](x)=\sigma(u(x)) v(x),\ \ \ \ \ \mbox{for }x \in\,\mathbb{R}^2.\]
With these notations, we say that an adapted process $u_{\epsilon} \in\,L^p(\Omega,C([0,T];H_\gamma))$ is a {\em mild solution} to equation  \eqref{SPDE} if it satisfies
\begin{equation}
u_{\epsilon}(t)=S_{\epsilon}(t)\varphi+\int_{0}^{t}S_{\epsilon}(t-s)B(u_{\epsilon}(s))ds+\int_{0}^{t}S_{\epsilon}(t-s)\Sigma(u_{\epsilon}(s))d\mathcal{W}(s).\label{eq: mild solution}
\end{equation}

If we denote by  $M$ the {\em multiplication operator} defined by
\[M(\psi)\xi=\psi \xi,\ \ \ \ \psi \in\,H_\gamma,\ \ \ \xi \in\,\mathcal{S}^\prime_q,\]
we have
\[\Sigma(u)v=M(\sigma(u))v.\]

As in  \cite{CerraiFreidlin2019}, where the noise in equation \eqref{SPDE} was a smooth Wiener process $\mathcal{W}$, having  finite spectral measure $\mu$,  we are here interested in studying the limiting behavior of $u_\e$, as $\e\to 0$,  in the space $L^p(\Omega;C([0,T];H_\gamma))$.  The limiting process will be the solution $\bar{u}$ of the following 
 SPDE on the graph $\Gamma$ 
\begin{equation}
\begin{cases}
\partial_{t}\bar{u}(t,z,k)=\bar{L}\bar{u}(t,z,k)+b(\bar{u}(t,z,k))+\sigma(\bar{u}(t,z,k))\partial_{t}\bar{\mathcal{W}}(t,z,k),\label{eq: SPDE on graph}\\
\bar{u}(0,z,k) =\varphi^{\wedge}(z,k),\qquad(z,k)\in\Gamma, 
\end{cases}
\end{equation}
where $\bar{\mathcal{W}}$ is the Wiener process on the graph $\Gamma$
corresponding to $\mathcal{W}$, as defined in \eqref{cx6}. We say $\bar{u}$ is a mild solution
to (\ref{eq: SPDE on graph}) if it is an adapted process in $L^p(\Omega;C([0,T];\bar{H}_\gamma))$ that satisfies the integral equation
\begin{equation}
\bar{u}(t)=\bar{S}(t)\varphi^{\wedge}+\int_{0}^{t}\bar{S}(t-s)B(\bar{u}(s))ds+\int_{0}^{t}\bar{S}(t-s)\Sigma(\bar{u}(s))d\bar{\mathcal{W}}(s).\label{eq: mild solution to the limit equation}
\end{equation}

\subsection{The semigroups $S_\e(t)$ and $\bar{S}(t)$}
\label{section3.1}

Here, we investigate the properties of the semigroups
$S_{\epsilon}(t)$ and their limit $\bar{S}(t)$. Firstly, we
review a few results obtained in previous works, where the following condition on the Hamiltonian $H$ is assumed.
\begin{assumption}\label{Assumption 2}

For any $(z,k)\in\Gamma$, we assume that
\[
\frac{dT_{k}(z)}{dz}\neq0.
\]

\end{assumption}

In \cite[Theorem A.2]{CerraiFreidlin2019} it
is shown that under Assumption
\ref{Assumption 2}, for any $u\in C_{b}(\rtwo)$, $x\in\rtwo$ and 
$0<\tau\leq T$ 
\begin{equation}
\lim_{\epsilon\rightarrow0}\sup_{t\in[\tau,T]}\vert S_{\epsilon}(t)u(x)-\bar{S}(t)^\vee u(x)\vert=0.\label{eq: semigroups weak convergence pointwise}
\end{equation}
Furthermore, in \cite[Corollary B.1]{CerraiFreidlin2019} it is shown 
that for any $u\in H_{\gamma}$ and $0<\tau\leq T$
\begin{align}
\lim_{\epsilon\rightarrow0}\sup_{t\in[\tau,T]}\left|S_{\epsilon}(t)u-\bar{S}(t)^{\vee}u\right|_{H_{\gamma}}^{2}=\lim_{\epsilon\rightarrow0}\sup_{t\in[\tau,T]}\left|\left(S_{\epsilon}(t)u\right)^{\wedge}-\bar{S}(t)u^{\wedge}\right|_{\bar{H}_{\gamma}}^{2}= & 0.\label{eq: semigroups weak convergence weighted space}
\end{align}

Suppose $G_{\epsilon}(t,x,y)$ is the kernel corresponding to $S_{\epsilon}(t)$,
i.e. 
\[
S_{\epsilon}(t)u(x)=\int_{\rtwo}G_{\epsilon}(t,x,y)u(y)dy,\ \ \ \ \ x \in\,\mathbb{R}^2.
\]
Limit (\ref{eq: semigroups weak convergence pointwise}) implies that
for any fixed $(t,x)$, kernels $G_{\epsilon}(t,x,\cdot)$ converge weakly
to some $\bar{G}(t,x,\cdot)$, which satisfies that
\[
\bar{S}(t)^{\vee}u(x)=\int_{\rtwo}\bar{G}(t,x,y)u(y)dy.
\]

Next, we determine the weighted space $H_{\gamma}$, on which the
semigroups $S_{\epsilon}(t)$ and $\bar{S}^\vee(t)$ are bounded. To determine the weight
$\gamma$, we have the following result from \cite[Proposition 4.1]{CerraiFreidlin2019}.
\begin{prop}
\label{prop: semigroup on the weight}There exists a strictly positive
decreasing function $h\in C^{2}([0,\infty))$ and a constant $C\geq 0$, such that the function
$\gamma:\Gamma\rightarrow(0,\infty)$ defined by $\gamma(z,k)=h(z)$,
for every $(z,k)\in\Gamma$, satisfies 
\begin{equation}
\int_{\rtwo}G_{\epsilon}(t,x,y)\gamma^{\vee}(x)dx\leq e^{C t}\gamma^{\vee}(y),\ \ \ \ y \in\,\rtwo,\label{eq: kernels applied to the weight}
\end{equation}
for every $t>0$. Moreover, for the same constant $C$, we
have that 
\begin{equation}
\vert S_{\epsilon}(t)u\vert_{H_{\gamma}}^{2}\leq e^{C t}\vert u\vert_{H_{\gamma}}^{2}.\label{eq: bounded semigroups}
\end{equation}
\end{prop}

\begin{rem}
The constant $C$ in Proposition \ref{prop: semigroup on the weight}
is independent of $\epsilon$. Therefore, by (\ref{eq: semigroups weak convergence pointwise})
and (\ref{eq: semigroups weak convergence weighted space}), for the
limit semigroup $\bar{S}(t)$, we also have that 
\begin{equation}
\int_{\rtwo}\bar{G}(t,x,y)\gamma^{\vee}(x)dx\leq e^{C t}\gamma^{\vee}(y),\ \ \ \ y \in\,\rtwo,
\label{eq: limit kernel applied to the weight}
\end{equation}
and 
\begin{equation}
\vert\bar{S}(t)^{\vee}u\vert_{H_{\gamma}}^{2}\leq e^{C t}\vert u\vert_{H_{\gamma}}^{2}\label{eq: bounded limit semigroups}
\end{equation}
for the same constant $C$. Throughout the rest of the paper,
we will always assume $\gamma$ to be a weight that satisfies (\ref{eq: kernels applied to the weight})-(\ref{eq: bounded limit semigroups})
as proved in Proposition \ref{prop: semigroup on the weight}.
\end{rem}

In addition to the weak convergence of the kernels $G_{\epsilon}(t,x,y)$
to $\bar{G}(t,x,y)$, we are now proving  the following uniform upper bound
to the kernels $G_{\epsilon}(t,x,y)$.
\begin{thm}
\label{thm: Kernels upper bound}Suppose the Hamiltonian $H$ satisfies
Assumption \ref{Assumption 1}. Then,  there exists a constant $C>0$ independent
of $\epsilon$ such that

\begin{equation}
G_{\epsilon}(t,x,y)\leq\frac{C}{t}\exp\left(-\frac{(\sqrt{H(y)+1}-\sqrt{H(x)+1})^{2}}{4Ct}\right),\label{eq: kernels pointwise upper bound}
\end{equation}
for any $(t,x,y)\in(0,T]\times\rtwo\times\rtwo$. Due to the weak convergence
of $G_{\epsilon}(t,x,y)$ to $\bar{G}(t,x,y)$, as $\epsilon\rightarrow0$,
the   same point-wise upper bound
as in (\ref{eq: kernels pointwise upper bound}) is valid for $\bar{G}(t,x,y)$.
\end{thm}

Before proving Theorem \ref{thm: Kernels upper bound}, we introduce some notation and prove a preliminary lemma.

To this purpose, we define $\psi(x)=\alpha\sqrt{H(x)+1}$, where the constant $\alpha\in\mathbb{R}$ is 
to be determined later. Since $\vert\nabla H(x)\vert\leq C\vert x\vert$
and $H(x)+1\geq C\vert x\vert^{2}$, we have that
\[
\left|\nabla\psi(x)\right|=\frac{\alpha\,\left|\nabla H(x)\right|}{2\sqrt{H(x)+1}}\leq\alpha\, C
\]
for some $C>0$. Now, for any $\e>0$ we consider the linear problem 
\begin{equation}
\begin{cases}
\partial_{t}z_{\epsilon}(t,x)=\frac{1}{2}\Delta z_{\epsilon}(t,x)+\frac{1}{\epsilon}\langle\nabla^{\perp}H(x),\nabla z_{\epsilon}(t,x)\rangle,\\
z_{\epsilon}(0,x)=z_{0}(x),
\end{cases}\label{eq: PDE}
\end{equation}
whose solution has representation 
\[z_{\epsilon}(t,x)=\int_{\rtwo}G_{\epsilon}(t,x,y)z_{0}(y)dy.\]

Now, we introduce  the transformed kernel 
\[
G_{\epsilon}^{T}(t,x,y):=e^{-\psi(x)}G_{\epsilon}(t,x,y)e^{\psi(y)},
\]
and we define
\[
z_{\epsilon}^{T}(t,x):=\int_{\rtwo}G_{\epsilon}^{T}(t,x,y)z_{0}(y)dy.
\]
The following result holds.
\begin{lem}
For any $p\geq1$, we have  
\begin{equation}
\frac{d}{dt}\Vert z_{\epsilon}^{T}(t,\cdot)\Vert_{L^{2p}}^{2p}\leq p^{2}\alpha^{2}C^{2}\Vert z_{\epsilon}^{T}(t,\cdot)\Vert_{L^{2p}}^{2p}-\Vert\nabla(z_{\epsilon}^{T}(t,\cdot))^{p}\Vert_{L^2}^{2}.\label{eq: kernel upper bound inequality}
\end{equation}
\end{lem}

\begin{proof}
By the definition of $z_{\epsilon}^{T}$ and $G_{\epsilon}^{T}$
\begin{align*}
\frac{d}{dt}\Vert z_{\epsilon}^{T}(t,\cdot)\Vert_{L^{2p}}^{2p} & =2p\int_{\rtwo}z_{\epsilon}^{T}(t,x)^{2p-1}\left(\int_{\rtwo}\frac{d}{dt}G_{\epsilon}^{T}(t,x,y)z_{0}(y)dy\right)dx\\
 & =2p\int_{\rtwo}z_{\epsilon}^{T}(t,x)^{2p-1}\left(\int_{\rtwo}e^{-\psi(x)+\psi(y)}\frac{1}{2}\Delta_{x}G_{\epsilon}^{T}(t,x,y)z_{0}(y)dy\right)dx\\
 & \quad +2p\int_{\rtwo}z_{\epsilon}^{T}(t,x)^{2p-1}\left(\int_{\rtwo}e^{-\psi(x)+\psi(y)}\frac{1}{\epsilon}\langle\nabla^{\perp}H(x),\nabla_{x}G_{\epsilon}^{T}(t,x,y)\rangle z_{0}(y)dy\right)dx.
\end{align*}
If we integrate by part  
\begin{align*}
\frac{d}{dt}\Vert z_{\epsilon}^{T}(t,\cdot)\Vert_{L^{2p}}^{2p} & =-2p\int_{\rtwo}\frac{1}{2}\left\langle \nabla\left(z_{\epsilon}^{T}(t,x)^{2p-1}e^{-\psi(x)}\right),\nabla\left(z_{\epsilon}^{T}(t,x)e^{\psi(x)}\right)\right\rangle dx\\
 & \quad+2p\int_{\rtwo}z_{\epsilon}^{T}(t,x)^{2p-1}e^{-\psi(x)}\frac{1}{\epsilon}\left\langle \nabla^{\perp}H(x),\nabla\left(z_{\epsilon}^{T}(t,x)e^{\psi(x)}\right)\right\rangle dx\\
 & =p\int_{\rtwo}z_{\epsilon}^{T}(t,x)^{2p}\vert\nabla\psi(x)\vert^{2}dx-(2p-2)\int_{\rtwo}\langle\nabla(z_{\epsilon}^{T}(t,x))^{p},\nabla\psi(x)\rangle z_{\epsilon}^{T}(t,x)^{p}dx\\
 & \quad-\frac{2p-1}{p}\int_{\rtwo}\vert\nabla(z_{\epsilon}^{T}(t,x))^{p}\vert^{2}dx+2p\int_{\rtwo}z_{\epsilon}^{T}(t,x)^{2p}\frac{1}{\epsilon}\langle\nabla^{\perp}H(x),\nabla\psi(x)\rangle dx\\
 & \quad +2p\int_{\rtwo}z_{\epsilon}^{T}(t,x)^{2p-1}\frac{1}{\epsilon}\langle\nabla^{\perp}H(x),\nabla z_{\epsilon}^{T}(t,x)\rangle dx.\\
 & =:I_{1}+I_{2}+I_{3}+I_{4}+I_{5}.
\end{align*}
The definition of $\nabla^{\perp}H(x)$ and $\psi$ clearly implies that $I_{4}=0$. Moreover,  since $\text{div} \nabla^{\perp}H=0$, we have
\[I_5= \frac{1}{\epsilon} \int_{\rtwo}\langle\nabla^{\perp}H(x),\nabla \left(z_{\epsilon}^{T}(t,x)^{2p}\right)\rangle dx=0.\]
Since $\vert\nabla\psi(x)\vert\leq\alpha C$
\begin{align*}
I_{2}+I_{3} & =-\int_{\rtwo}\vert\nabla(z_{\epsilon}^{T}(t,x))^{p}\vert^{2}dx+p(p-1)\int_{\rtwo}\vert\nabla\psi(x)\vert^{2}z_{\epsilon}^{T}(t,x)^{2p}dx\\
 & \quad-\frac{p-1}{p}\int_{\rtwo}\vert\nabla(z_{\epsilon}^{T}(t,x))^{p}+p\nabla\psi(x)z_{\epsilon}^{T}(t,x)^{p}\vert^{2}dx\\
 & \leq p(p-1)\alpha^{2}C^{2}\Vert z_{\epsilon}^{T}(t,\cdot)\Vert_{L^{2p}}^{2p}-\Vert\nabla(z_{\epsilon}^{T}(t,\cdot))^{p}\Vert_{L^2}^{2}.
\end{align*}
Together with 
\[
I_{1}\leq p\alpha^{2}C^{2}\Vert z_{\epsilon}^{T}(t,\cdot)\Vert_{L^{2p}}^{2p},
\]
we complete the proof.
\end{proof}

{\em Proof of Theorem \ref{thm: Kernels upper bound}.}
If we apply Nash's inequality and inequality (\ref{eq: kernel upper bound inequality})
as in \cite[Lemma 1.4]{FabesStroock1986}, we can deduce that 
\begin{equation}
\Vert z_{\epsilon}^{T}(t,\cdot)\Vert_{L^{\infty}}\leq\frac{C}{t^{1/2}}\exp(C\alpha^{2}t)\Vert z_{0}\Vert_{L^2}.\label{eq: inifity leq 2}
\end{equation}
The dual equation to (\ref{eq: PDE}) only changes the sign of the
first order coefficient $\nabla^{\perp}H(x)$, which means (\ref{eq: inifity leq 2})
is also true for the dual equation. By duality, there is
\[
\Vert z_{\epsilon}^{T}(t,\cdot)\Vert_{L^2}\leq\frac{C}{t^{1/2}}\exp(C\alpha^{2}t)\Vert z_{0}\Vert_{L^1}.
\]
Together with \eqref{eq: inifity leq 2}, this implies that
\[
\Vert z_{\epsilon}^{T}(t,\cdot)\Vert_{L^{\infty}}\leq\frac{C}{t}\exp(C\alpha^{2}t)\Vert z_{0}\Vert_{L^1}.
\]
By the definition of $z_{\epsilon}^{T}(t,x)$, we obtain that
\[
G_{\epsilon}^{T}(t,x,y)\leq\frac{C}{t}\exp(C\alpha^{2}t),
\]
and hence 
\[
G_{\epsilon}(t,x,y)\leq\frac{C}{t}\exp\left(C\alpha^{2}t+\alpha\sqrt{H(x)+1}-\alpha\sqrt{H(y)+1}\right)
\]
for any $\alpha\in\mathbb{R}$, $t\in(0,\infty)$ and $x,y\in\rtwo$.
Here we can take $\alpha=\frac{\sqrt{H(y)+1}-\sqrt{H(x)+1}}{2Ct}$
to minimize the right-hand side to obtain
\[
G_{\epsilon}(t,x,y)\leq\frac{C}{t}\exp\left(-\frac{(\sqrt{H(y)+1}-\sqrt{H(x)+1})^{2}}{4Ct}\right).
\]
\qed
\begin{cor}
\label{cor1}
Given any compact subset $K\subset\rtwo$, there exist $C$ and $R$
depending on $K$ such that 
\begin{equation}
\label{cx1}
\sup_{x \in\,K}G_{\epsilon}(t,x,y)\leq\begin{cases}
\frac{C}{t} & \vert y\vert\leq R\\
\frac{C}{t}\exp\left(-\frac{\vert y\vert^{2}}{Ct}\right) & \vert y\vert>R
\end{cases}
\end{equation}
for any $t\in(0,\infty)$ and $y\in\rtwo$. Moreover, the limit $\bar{G}(t,x,y)$ satisfies the same upper bound as in \eqref{cx1}.
\end{cor}

\begin{proof}
Actually we always have $G_{\epsilon}(t,x,y)\leq\frac{C}{t}$. Then
since $H(x)$ is bounded for $x\in K$, by Assumption \ref{Assumption 1} we have that 
\[
G_{\epsilon}(t,x,y)\leq\frac{C}{t}\exp\left(-\frac{H(y)+H(x)+2-2\sqrt{H(y)+1}\sqrt{H(x)+1}}{4Ct}\right)\leq\frac{C}{t}\exp\left(-\frac{\vert y\vert^{2}}{Ct}\right)
\]
for large enough $\vert y\vert$.
\end{proof}
Now we consider the stochastic convolutions
\begin{equation}
\label{cx2}
\int_{0}^{t}S_{\epsilon}(t-s)\Sigma(u_{\epsilon}(s))d\mathcal{W}(s)
\end{equation}
and 
\begin{equation}
\label{cx2ForLimit}
\int_{0}^{t}\bar{S}(t-s)\Sigma(\bar{u}(s))d\bar{\mathcal{W}}(s),
\end{equation}
as in the definition of mild solutions, and show that they are well-defined in $H_\gamma$ and $\bar{H}_\gamma$, respectively,
when the spectral measure $\mu$ of the spatially homogeneous Wiener
process $\mathcal{W}$ has density function $m$ in $L^{p}(\rtwo)$, with $p\in(1,\infty)$.

To be more precise, as stated in the following lemma, 
we show that the semigroup $S_{\epsilon}(t)$ improves the regularity
of \eqref{cx2} following the proof of \cite[Proposition 4.1]{PeszatZabczyk1997}.

\begin{lem}
\label{lem: Semigroup property 1} Under Assumption \ref{Assumption 3}, $S_\e(t) M(\psi)$ are Hilbert-Schmidt operators from $\mathscr{S}_{q}'$ to $H_{\gamma}$, for all
 $\psi\in H_{\gamma}$. Moreover,
for each $T>0$ there exists a constant $C_T>0$ such that  
\[
\Vert S_{\epsilon}(t)M(\psi)\Vert_{L_{(HS)}(\mathscr{S}_{q}',H_{\gamma})}^{2}\leq C_T \Vert m\Vert_{L^p}\,t^{-(p-1)/p}\vert\psi\vert_{H_{\gamma}}^{2},\ \ \ \ \ t \in\,[0,T].
\]
\end{lem}

\begin{proof}
Let $\{v_{j}\}$ be an orthonormal basis of $L_{(s)}^{2}(\rtwo,dx)$.
Thanks  to \eqref{prop:spatially homogeneous WP property}, if we define $e_{j}=\widehat{v_{j}m^{1/2}}$, we have that $\{e_j\}_{j \in\,\mathbb{N}}$ is an orthonormal complete system
in $\mathscr{S}_{q}'$.
Then, for any $\psi \in\,H_\gamma$
\begin{align*}
I & :=\sum_{j=1}^{\infty}\vert S_{\epsilon}(t)\psi e_{j}\vert_{H_{\gamma}}^{2}=\sum_{j=1}^\infty\left|S_{\epsilon}(t)\left[\psi(\widehat{m^{1/2}}\ast\widehat{v_{j}})\right]\right|_{H_{\gamma}}^{2}\\
 & =\sum_{j=1}^{\infty}\int_{\rtwo}\left[\int_{\rtwo}G_{\epsilon}(t,x,y)\psi(y)(\widehat{m^{1/2}}\ast\widehat{v_{j}})(y)dy\right]^{2}\gamma^{\vee}(x)dx\\
 & \leq\int_{\rtwo}\int_{\rtwo}\left|\widehat{m^{1/2}}\ast(G_{\epsilon}(t,x,\cdot)\psi)(y)\right|^{2}dy\,\gamma^{\vee}(x)dx\\
 & =\int_{\rtwo}\int_{\rtwo}\vert m^{1/2}(y)\vert^{2}\vert\widehat{(G_{\epsilon}(t,x,\cdot)\psi)}(y)\vert^{2}dy\,\gamma^{\vee}(x)dx\\
 & \leq\Vert m\Vert_{L^p}\int_{\rtwo}\Vert\widehat{(G_{\epsilon}(t,x,\cdot)\psi)}\Vert_{L^{2p^{\ast}}}^{2}\,\gamma^{\vee}(x)dx,
\end{align*}
where $p^{\ast}$ is the H\"older conjugate of $p$. The Hausdorff-Young
inequality implies that $\Vert\widehat{(G_{\epsilon}(t,x,\cdot)\psi)}\Vert_{L^{2p^{\ast}}}\leq\Vert(G_{\epsilon}(t,x,\cdot)\psi)\Vert_{L^{2p/(p+1)}}$
and we obtain 
\begin{align*}
I & \leq\Vert m\Vert_{L^p}\int_{\rtwo}\Vert(G_{\epsilon}(t,x,\cdot)\psi)\Vert_{L^{2p/(p+1)}}^{2}\gamma^{\vee}(x)dx\\
 & =\Vert m\Vert_{L^p}\int_{\rtwo}\left[\int_{\rtwo}\vert G_{\epsilon}(t,x,y)\psi(y)\vert^{2p/(p+1)}dy\right]^{(p+1)/p}\gamma^{\vee}(x)dx.
\end{align*}
By Theorem \ref{thm: Kernels upper bound}, we have that 
\begin{align*}
I & \leq C\Vert m\Vert_{L^p}t^{-(p-1)/p}\int_{\rtwo}\left[\int_{\rtwo}G_{\epsilon}(t,x,y)\vert\psi(y)\vert^{2p/(p+1)}dy\right]^{(p+1)/p}\gamma^{\vee}(x)dx.
\end{align*}
Since $2p/(p+1)\leq2$ and $G_{\epsilon}(t,x,y)dy$ is a probability
measure, 
\[
\left[\int_{\rtwo}G_{\epsilon}(t,x,y)\vert\psi(y)\vert^{2p/(p+1)}dy\right]^{(p+1)/p}\leq\left[\int_{\rtwo}G_{\epsilon}(t,x,y)\vert\psi(y)\vert^{2}dy\right],
\]
and then, using Proposition \ref{prop: semigroup on the weight}, we conclude
\begin{align*}
I & \leq C\Vert m\Vert_{L^p}t^{-(p-1)/p}\int_{\rtwo}\left[\int_{\rtwo}G_{\epsilon}(t,x,y)\vert\psi(y)\vert^{2}dy\right]\gamma^{\vee}(x)dx\\
 & =C\Vert m\Vert_{L^p}t^{-(p-1)/p}\int_{\rtwo}\int_{\rtwo}G_{\epsilon}(t,x,y)\gamma^{\vee}(x)dx\vert\psi(y)\vert^{2}dy\\
 & \leq C\Vert m\Vert_{L^p}t^{-(p-1)/p}e^{CT}\int_{\rtwo}\vert\psi(y)\vert^{2}\gamma^{\vee}(y)dy\\
 & =Ce^{CT}\Vert m\Vert_{L^p}t^{-(p-1)/p}\vert\psi\vert_{H_{\gamma}}^{2}.
\end{align*}
\end{proof}
Now we consider the limit semigroup $\bar{S}(t)$ and show that an analogous result holds.
\begin{lem}
\label{lem: limit semigroup property 1}Under Assumption \ref{Assumption 3}, $\bar{S}(t) M(\psi)$ are Hilbert–Schmidt operators from $\bar{\mathscr{S}}_{q}'$ to $\bar{H}_{\gamma}$. For each $T>0$ there exists a constant $C_T>0$ such that for all
 $\psi\in\bar{H}_{\gamma}$  
\[
\Vert\bar{S}(t)M(\psi)\Vert_{L_{(HS)}(\bar{\mathscr{S}_{q}'},\bar{H}_{\gamma})}^{2}\leq C_T\Vert m\Vert_{L^p}\,t^{-(p-1)/p}\vert\psi\vert_{\bar{H}_{\gamma}}^{2},\ \ \ \ \ t \in\,[0,T].
\]
\end{lem}

\begin{proof}
We have
\[I :=\sum_{j=1}^{\infty}\vert\bar{S}(t)\psi e_{j}^{\wedge}\vert_{\bar{H}_{\gamma}}^{2}=\sum_{j=1}^\infty\left|\bar{S}(t)\left[\psi(\widehat{m^{1/2}}\ast\widehat{v_{j}})^{\wedge}\right]\right|_{\bar{H}_{\gamma}}^{2}.\]
Then by Proposition \ref{prop:weightedSpaces} and the definition of $\bar{S}(t)^\vee$ and $\bar{G}(t,x,y)$,
\begin{align*}
I &  =\sum_{j=1}^{\infty}\left|\bar{S}(t)^{\vee}\left[\psi^{\vee}(\widehat{m^{1/2}}\ast\widehat{v_{j}})\right]\right|_{H_{\gamma}}^{2}\\
 & =\sum_{j=1}^\infty\int_{\rtwo}\left[\int_{\rtwo}\bar{G}(t,x,y)\psi^{\vee}(y)(\widehat{m^{1/2}}\ast\widehat{v_{j}})(y)dy\right]^{2}\gamma^{\vee}(x)dx\\
 & \leq\int_{\rtwo}\int_{\rtwo}\left|\widehat{m^{1/2}}\ast(\bar{G}(t,x,\cdot)\psi^{\vee})(y)dy\right|^2\gamma^{\vee}(x)dx\\
 & =\int_{\rtwo}\int_{\rtwo}\vert m^{1/2}(y)\vert^{2}\vert\widehat{(\bar{G}(t,x,\cdot)\psi^{\vee})}(y)\vert^{2}dy\gamma^{\vee}(x)dx\\
 & \leq\Vert m\Vert_{L^p}\int_{\rtwo}\Vert\widehat{(\bar{G}(t,x,\cdot)\psi^{\vee})}\Vert_{L^{2p^{\ast}}}^{2}\gamma^{\vee}(x)dx.
\end{align*}
Now, with the same arguments used in   the proof of Lemma \ref{lem: Semigroup property 1},
using (\ref{eq: limit kernel applied to the weight}) and the bound $\bar{G}(t,x,y)\leq C\,t^{-1}$,
we have that 
\[
I\leq Ce^{CT}\Vert m\Vert_{L^p}t^{-(p-1)/p}\vert\psi^\vee\vert_{H_{\gamma}}^{2}=Ce^{CT}\Vert m\Vert_{L^p}t^{-(p-1)/p}\vert\psi\vert_{\bar{H}_{\gamma}}^{2},
\]
where the last equality follows from Proposition \ref{prop:weightedSpaces}.
\end{proof}

Using classical arguments, in section \ref{ExistenceAndUniqueness} we will show that Lemma \ref{lem: Semigroup property 1} and Lemma \ref{lem: limit semigroup property 1} imply that SPDEs \eqref{SPDE} and \eqref{eq: SPDE on graph} admit a unique mild solution. 

Next, to prove the convergence of mild solutions $u_{\epsilon}$ of equations
(\ref{SPDE}) to the mild solution $\bar{u}$ of equation (\ref{eq: SPDE on graph}), we show that the three terms in the definition of mild solutions \eqref{eq: mild solution} converge to that of \eqref{eq: mild solution to the limit equation}. Among these three terms, the most difficult one is the convergence of the stochastic integrals \eqref{cx2} to \eqref{cx2ForLimit}, for which we will need the following approximation result.
\begin{lem}
\label{lem: convergenc of the infinite sum}Given any $\psi\in H_{\gamma}$,
for any fixed $0<\tau < T$
\begin{equation}
\label{cX2}
\lim_{\epsilon\rightarrow0}\sup_{t \in\,[\tau,T]}\,\sum_{j=1}^{\infty}\left|(S_{\epsilon}(t)-\bar{S}(t)^{\vee})\left(\psi e_{j}\right)\right|_{H_{\gamma}}^{2}=0.
\end{equation}
\end{lem}

\begin{proof}
We show that for any given $\delta>0$, there exists $\epsilon_{\delta}>0$
such that for any $0<\epsilon\leq\epsilon_{\delta}$, 
\begin{equation}
\label{cX2equivalent}
\sum_{j=1}^{\infty}\left|(S_{\epsilon}(t)-\bar{S}(t)^{\vee})\left(\psi e_{j}\right)\right|_{H_{\gamma}}^{2}\leq\delta,\ \ \ \ t \in\,[\tau,T].
\end{equation}
The spectral measure $m$ belongs to $L^{p}(\mathbb{R}^2)$, for $p\in[1,\infty)$, which
means that $m^{1/2}\in L^{2p}(\mathbb{R}^2)$. Given any $\eta>0$, we write $m=m_{1}+m_{2}$,
where 
\[m_{1}=m1_{\{m<\eta\}},\ \ \ \ m_{2}=m1_{\{m\geq\eta\}}.\]
Then $m^{1/2}=m_{1}^{1/2}+m_{2}^{1/2}$ and 
\begin{align*}
I & =\sum_{j=1}^{\infty}\left|(S_{\epsilon}(t)-\bar{S}(t)^{\vee})\left(\psi\widehat{v_{j}m^{1/2}}\right)\right|_{H_{\gamma}}^{2}\\
 & =\sum_{j=1}^{\infty}\left|(S_{\epsilon}(t)-\bar{S}(t)^{\vee})\left(\psi\widehat{v_{j}m_{1}^{1/2}}\right)+(S_{\epsilon}(t)-\bar{S}(t)^{\vee})\left(\psi\widehat{v_{j}m_{2}^{1/2}}\right)\right|_{H_{\gamma}}^{2}\\
 & \leq 2\sum_{j=1}^{\infty}\left|(S_{\epsilon}(t)-\bar{S}(t)^{\vee})\left(\psi\widehat{v_{j}m_{1}^{1/2}}\right)\right|_{H_{\gamma}}^{2}+2\sum_{j=1}^{\infty}\left|(S_{\epsilon}(t)-\bar{S}(t)^{\vee})\left(\psi\widehat{v_{j}m_{2}^{1/2}}\right)\right|_{H_{\gamma}}^{2}\\
 & =:I_{1}(\e,t,\eta)+I_{2}(\e,t,\eta).
\end{align*}

For the first term, due to \eqref{cx1}, we have  
\begin{align*}
I_{1}(\e,t,\eta) & \leq\sum_{j=1}^\infty \int_{\rtwo}\left[\int_{\rtwo}(G_{\epsilon}(t,x,y)-\bar{G}(t,x,y))\psi(y)(\widehat{m_{1}^{1/2}}\ast\widehat{v_{j}})(y)dy\right]^{2}\gamma^{\vee}(x)dx\\
 & \leq\int_{\rtwo}\int_{\rtwo}\vert m_{1}^{1/2}(y)\vert^{2}\vert\widehat{((G_{\epsilon}(t,x,\cdot)-\bar{G}(t,x,\cdot))\psi)}(y)\vert^{2}dy\,\gamma^{\vee}(x)dx\\
 & \leq\eta\int_{\mathbb{R}^2}\int_{\rtwo}\vert(G_{\epsilon}(t,x,y)-\bar{G}(t,x,y))\psi(y)\vert^{2}dy\,\gamma^{\vee}(x)dx\\
 & \leq C\eta t^{-1}\int_{\mathbb{R}^2}\int_{\rtwo}G_{\epsilon}(t,x,y)\vert\psi(y)\vert^{2}+\bar{G}(t,x,y)\vert\psi(y)\vert^{2}dy\,\gamma^{\vee}(x)dx\\
 & = C\eta t^{-1}\int_{\mathbb{R}^2}\int_{\rtwo}\left[G_{\epsilon}(t,x,y)\gamma^{\vee}(x)+\bar{G}(t,x,y)\gamma^{\vee}(x)\right]dx\vert\psi(y)\vert^{2}dy.
 \end{align*}
Then, thanks to \eqref{eq: kernels applied to the weight} and \eqref{eq: limit kernel applied to the weight}, we get
\[I_{1}(\e,t,\eta)   \leq C\eta t^{-1}2e^{Ct}\int_{\rtwo}\vert\psi(y)\vert^{2}\gamma^{\vee}(y)dy =C\eta t^{-1}2e^{Ct}\vert\psi\vert_{H_{\gamma}}^{2}.\]
This means that we can fix $\eta_\delta=\eta(\delta,\tau,T,\psi)>0$ such that 
\begin{equation}
\label{cx4}
\sup_{\e>0}\,\sup_{ t \in\,[\tau,T]}\,I_{1}(\e,t,\eta_\delta)\leq \frac{\delta}2.
\end{equation}

Now, concerning the second term $I_2(\e,t,\eta)$, we have
\begin{align*}
\sum_{j=1}^{\infty}\left|\psi\widehat{v_{j}m_{2}^{1/2}}\right|_{H_{\gamma}}^{2} & =\sum_{j=1}^{\infty}\int_{\rtwo}\left|\psi(x)\widehat{v_{j}m_{2}^{1/2}}(x)\right|^2\gamma^{\vee}(x)dx\\
 & =(2\pi)^{-2}\int_{\rtwo}\sum_{j=1}^{\infty}\left|\int_{\mathbb{R}^2}\exp(i\xi\cdot x)\psi(x)v_{j}(\xi)m_{2}^{1/2}(\xi)d\xi\right|^{2}\gamma^{\vee}(x)dx\\
 & \leq (2\pi)^{-2}\int_{\rtwo}\int_{\rtwo}\vert\exp(i\xi\cdot x)\psi(x)\vert^{2}m_{2}(\xi)d\xi\,\gamma^{\vee}(x)dx\\
 & =(2\pi)^{-2}\Vert m_{2}\Vert_{L^{1}}\vert\psi\vert_{H_{\gamma}}^{2}.
\end{align*}
Then, since
$\Vert m_{2}\Vert_{L^{1}}\leq\Vert m\Vert_{L^{p}}/\eta^{p-1}$, if we take $\eta=\eta_\delta$ we get
\[\sum_{j=1}^{\infty}\left|\psi\widehat{v_{j}m_{2}^{1/2}}\right|_{H_{\gamma}}^{2}\leq (2\pi)^{-2} \eta_\delta^{-(p-1)}\Vert m\Vert_{L^p}|\psi|^2_{H_\gamma}.\]
Due to \eqref{eq: bounded semigroups} and \eqref{eq: bounded limit semigroups}, this implies that
 we can  choose $N_\delta$ large enough such that 
 \begin{equation}
 \label{cX2dot1}
 \sup_{\e>0}\,\sup_{t \in\,[\tau,T]}\,\sum_{j=N_\delta+1}^{\infty}\left|(S_{\epsilon}(t)-\bar{S}(t)^{\vee})\left(\psi\widehat{v_{j}m_{2}^{1/2}}\right)\right|_{H_{\gamma}}^{2}\leq \frac \delta 4.
 \end{equation}
Moreover, by (\ref{eq: semigroups weak convergence weighted space}), we
can choose $0<\epsilon_{\delta}$ small enough such that 
\begin{equation}
\label{cX2dot2}
\sup_{t \in\,[\tau,T]}
\sum_{j=1}^{N_\delta}\left|(S_{\epsilon}(t)-\bar{S}(t)^{\vee})\left(\psi\widehat{v_{j}m_{1}^{1/2}}\right)\right|_{H_{\gamma}}^{2}\leq\frac \delta 4
\end{equation}
for any $\e\leq \e_\delta$. These two inequalities \eqref{cX2dot1} and \eqref{cX2dot2}, together with \eqref{cx4}, imply \eqref{cX2equivalent}. 
\end{proof}

\subsection{Existence and uniqueness }\label{ExistenceAndUniqueness}
\label{section3.2}
Here we state the existence and uniqueness of mild solutions to SPDEs \eqref{SPDE} and \eqref{eq: SPDE on graph} using Lemma \ref{lem: Semigroup property 1} and Lemma \ref{lem: limit semigroup property 1}. We state it in the following theorem.

\begin{thm}
\label{thm: ExistenceUniqueness}Suppose the Hamiltonian $H$ satisfies
Assumption \ref{Assumption 1}, coefficients $b$ and $\sigma$ satisfy Assumption \ref{Assumption 4}. We assume that the spectral
measure of the spatially homogeneous Wiener process $\mathcal{W}(t)$
satisfies Assumption \ref{Assumption 3}, i.e. there is a density function $m(x)=d\mu /dx\in L^p(\rtwo)$ for some $p\in(1,\infty)$. Given $q\geq 1$, $\mathscr{H}_{q}:=L^{q}(\Omega;C([0,T];H_{\gamma}))$ and $\bar{\mathscr{H}}_{q}:=L^{q}(\Omega;C([0,T];\bar{H}_{\gamma}))$ are Banach spaces with norms 
\[
\Vert u\Vert_{\mathscr{H}_{q}}=\left(\mathbb{E}\sup_{t\in[0,T]}\vert u(t)\vert_{H_{\gamma}}^{q}\right)^{1/q},\qquad \Vert \bar{u}\Vert_{\bar{\mathscr{H}}_{q}}=\left(\mathbb{E}\sup_{t\in[0,T]}\vert \bar{u}(t)\vert_{\bar{H}_{\gamma}}^{q}\right)^{1/q},
\]
respectively. Then for any $\e>0$ and $q>2p$, there is a unique mild solution $u_\e$ to \eqref{SPDE} satisfying that
\begin{equation}
\label{bound1}
\sup_{\e \in\,(0,1)}\Vert u_\e \Vert_{\mathscr{H}_{q}}^q \leq C_{T, q}\left(1+|\varphi|_{H_\gamma}^q\right).
\end{equation}
Moreover, there is also a unique mild solution $\bar{u}$ to  \eqref{eq: SPDE on graph} satisfying 
\begin{equation}
\label{bound2}
\Vert \bar{u}\Vert_{\bar{\mathscr{H}}_{q}}^q \leq C_T\left(1+|\varphi^\wedge|^q_{\bar{H}_\gamma}\right).
\end{equation}
\end{thm}

\begin{rem*}
As discussed in \cite{CerraiFreidlin2019}, the existence and uniqueness of the mild solutions stated in Theorem \ref{ExistenceAndUniqueness} is also true if the spectral measure $\mu$ is finite. Together, Theorem \ref{ExistenceAndUniqueness} is actually true when the spectral measure can be written as $\mu=\mu_1+\mu_2$, where $\mu_1$ is a finite measure and $\mu_2$ has density function $m\in L^p(\rtwo)$ for some $p\in(1,\infty)$. 
\end{rem*}
The proof of Theorem \ref{ExistenceAndUniqueness} follows the arguments in \cite{PeszatZabczyk1997}, which is essentially to show that all terms in the definition of the mild solutions \eqref{eq: mild solution} and \eqref{eq: mild solution to the limit equation} are contraction mappings on Banach spaces $\mathscr{H}_{q}$ and $\bar{\mathscr{H}}_{q}$, respectively. The condition that $q>2p$ is required for the construction of contraction mappings. By H\"older's inequality, actually the mild solutions are in $\mathscr{H}_{q}$ and $\bar{\mathscr{H}}_{q}$ for any $q\geq1$. Here we omit the detailed proof of Theorem \ref{thm: ExistenceUniqueness}, since it is standard. 

\subsection{Convergence of the mild solutions}
\label{section3.3}
In this section, we study the convergence of $u_{\epsilon}$ to $\bar{u}$.
The main result of this section is stated in the following theorem.
\begin{thm}
\label{thm: Convergence of SPDE}Suppose the Hamiltonian $H$ satisfies
Assumption \ref{Assumption 1} and \ref{Assumption 2}, coefficients $b$ and $\sigma$ satisfy Assumption \ref{Assumption 4} and the spectral
measure of the spatially homogeneous Wiener process $\mathcal{W}(t)$
satisfies Assumption \ref{Assumption 3}. Let $u_{\epsilon}$
be the unique mild solution to (\ref{SPDE}) and $\bar{u}$
be the unique mild solution to (\ref{eq: SPDE on graph}), with the
same initial conditions $\varphi$ and $\varphi^\wedge$, respectively. Then, for any fixed $q \geq 1$ and $0<\tau<T$, we have that
\begin{equation}\label{ConvergSPDE}
\lim_{\epsilon\rightarrow0}\mathbb{E}\sup_{t\in[\tau,T]}\vert u_{\epsilon}(t)-\bar{u}(t)^{\vee}\vert_{H_{\gamma}}^{q}=\lim_{\epsilon\rightarrow0}\mathbb{E}\sup_{t\in[\tau,T]}\vert u_{\epsilon}(t)^{\wedge}-\bar{u}(t)\vert_{\bar{H}_{\gamma}}^{q}=0.
\end{equation}
\end{thm}

\begin{proof}
Without loss of generality, it is enough to prove \eqref{ConvergSPDE} for large enough $q>2p$.  For any fixed $0<\tau< T$ and $q>2p$, we denote by
\[
\Delta_{\epsilon,q}(\tau,t):=\mathbb{E}\sup_{s\in[\tau,t]}\vert u_{\epsilon}(s)-\bar{u}(s)^{\vee}\vert_{H_{\gamma}}^{q},\ \ \ \ \ t \in\,[\tau,T],\ \ \ \e>0.
\]
Then there is
\begin{align*}
u_{\epsilon}(s)-\bar{u}(s)^{\vee} & =\left[S_{\epsilon}(s)\varphi-\bar{S}(s)^{\vee}\varphi\right]\\
 & \quad+\left[\int_{0}^{s}S_{\epsilon}(s-r)B(u_{\epsilon}(r))dr-\left(\int_{0}^{s}\bar{S}(s-r)B(\bar{u}(r))dr\right)^\vee\right]\\
 & \quad+\left[\int_{0}^{s}S_{\epsilon}(s-r)\Sigma(u_{\epsilon}(r))d\mathcal{W}(r)-\left(\int_{0}^{s}\bar{S}(s-r)\Sigma(\bar{u}(r))d\bar{\mathcal{W}}(r)\right)^\vee\right]\\
 & =\left[S_{\epsilon}(s)\varphi-\bar{S}(s)^{\vee}\varphi\right]\\
 & \quad+\left[\int_{0}^{s}S_{\epsilon}(s-r)B(u_{\epsilon}(r))dr-\int_{0}^{s}\bar{S}(s-r)^{\vee}B(\bar{u}(r)^{\vee})dr\right]\\
 & \quad+\left[\int_{0}^{s}S_{\epsilon}(s-r)\Sigma(u_{\epsilon}(r))d\mathcal{W}(r)-\int_{0}^{s}\bar{S}(s-r)^{\vee}\Sigma(\bar{u}(r)^{\vee})d\mathcal{W}(r)\right]\\
 & =:I_{\epsilon,1}(s)+I_{\epsilon,2}(s)+I_{\epsilon,3}(s).
\end{align*}
Therefore, due to  Lemma \ref{lem: SPDE convergence second term}
and Lemma \ref{lem: SPDE convergence third term} below, 
we have
\begin{align*}
\Delta_{\epsilon,q}(\tau,t) & \leq\sum_{i=1}^{3}\mathbb{E}\sup_{s\in[\tau,t]}\left|I_{\epsilon,i}(s)\right|_{H_{\gamma}}^{q}\\
 & \leq C_{q,T}\int_{\tau}^{t}\Delta_{\epsilon,q}(\tau,s)ds+C_{q, T}\tau+\mathbb{E}\sup_{s\in[\tau,T]}\left|I_{\epsilon,1}(s)\right|_{H_{\gamma}}^{q}+H_{\epsilon,1}(\tau, T)+H_{\epsilon,2}(T),
\end{align*}
and, thanks to the Gr\"onwall lemma, this implies
\[\Delta_{\epsilon,q}(\tau,t)\leq C_{q, T}\left[\tau+\mathbb{E}\sup_{s\in[\tau,T]}\left|I_{\epsilon,1}(s)\right|_{H_{\gamma}}^{q}+H_{\epsilon,1}(\tau, T)+H_{\epsilon,2}(T)\right].\]
Firstly, it is enough to prove \eqref{ConvergSPDE} for small enough $\tau$. Hence, for any $\delta>0$ fixed, we can choose $\tau_\delta$ small enough so that $C_{q, T}\tau<\delta/2$ for every $\tau\leq \tau_\delta$.
Next, we notice that
by (\ref{eq: semigroups weak convergence weighted space}),
we have 
\[
\lim_{\epsilon\rightarrow0}\mathbb{E}\sup_{s\in[\tau,t]}\left|I_{\epsilon,1}(s)\right|_{H_{\gamma}}^{q}=\lim_{\epsilon\rightarrow0}\mathbb{E}\sup_{s\in[\tau,T]}\left|I_{\epsilon,1}(s)^\wedge\right|_{\bar{H}_{\gamma}}^{q}=0.
\]
Thus, thanks to \eqref{cx12} and \eqref{cx13}, we can find $\e_\delta>0$ such that 
\[C_{T, q}\left[\mathbb{E}\sup_{s\in[\tau,T]}\left|I_{\epsilon,1}(s)\right|_{H_{\gamma}}^{q}+H_{\epsilon,1}(\tau, T)+H_{\epsilon,2}(T)\right]<\delta/2,\]
for every $\e\leq \e_\delta$ and $0\leq \tau<T$. This clearly implies our theorem.
\end{proof}

\begin{lem}
\label{lem: SPDE convergence second term}For every $q\geq 1$ and for every $0<\tau<T$ there exists $C_{q, T}>0$ such that for every $0<\tau\leq t \leq T$
\begin{equation}
\mathbb{E}\sup_{s\in[0,t]}\vert I_{\epsilon,2}(s)\vert_{H_{\gamma}}^{q}\leq C_{q,T}\left(\int_{\tau}^{t}\mathbb{E}\sup_{r\in[\tau,s]}\vert u_{\epsilon}(r)-\bar{u}(r)^{\vee}\vert_{H_{\gamma}}^{q}ds+\tau\right)+H_{\e, 1}(\tau,T),\label{eq: convergence theorme second term}
\end{equation}
where $H_{\e, 1}(\tau,T)$ satisfies that
\begin{equation}
\label{cx12}
\lim_{\e\to 0} H_{\e, 1}(\tau,T)=0.
\end{equation}
\end{lem}

\begin{proof}
We have 
\begin{align*}
I_{\epsilon,2}(t) & =\int_{0}^{t}S_{\epsilon}(t-s)\left[B(u_{\epsilon}(s))-B(\bar{u}(s)^{\vee})\right]ds\\
 & \quad+\int_{0}^{t}\left[S_{\epsilon}(t-s)-\bar{S}(t-s)^{\vee}\right]B(\bar{u}(s)^{\vee})ds\\
 & =:J_{\epsilon,1}(t)+J_{\epsilon,2}(t).
\end{align*}
Then, since $|B(u)|_{H_\gamma}\leq c\left(1+|u|_{H_\gamma}\right)$,  for any $t, \tau >0$ we have that 
\begin{align}
\label{cx41}
\vert J_{\epsilon,1}(t)\vert_{H_{\gamma}}^{q} & \leq C_{q}\left|\int_{0}^{\tau}S_{\epsilon}(t-s)\left[B(u_{\epsilon}(s))-B(\bar{u}(s)^{\vee})\right]ds\right|_{H_{\gamma}}^{q}\nonumber\\
 & \quad+C_{q}\left|\int_{\tau}^{t}S_{\epsilon}(t-s)\left[B(u_{\epsilon}(s))-B(\bar{u}(s)^{\vee})\right]ds\right|_{H_{\gamma}}^{q}\nonumber\\
 & \leq C_{q,T}\int_{0}^{\tau}\left(1+\vert u_{\epsilon}(s)\vert_{H_{\gamma}}^{q}+\vert\bar{u}(s)^{\vee}\vert_{H_{\gamma}}^{q}\right)ds\\
 & \quad+C_{q,T}\int_{\tau}^{t}\vert u_{\epsilon}(s)-\bar{u}(s)^{\vee}\vert_{H_{\gamma}}^{q}ds\nonumber\\
 & \leq C_{q,T}\tau\sup_{s\in[0,T]}\left(1+\vert u_{\epsilon}(s)\vert_{H_{\gamma}}^{q}+\vert\bar{u}(s)^{\vee}\vert_{H_{\gamma}}^{q}\right)\nonumber\\
 & \quad+C_{q,T}\int_{\tau}^{t}\sup_{r\in[\tau,s]}\vert u_{\epsilon}(r)-\bar{u}(r)^{\vee}\vert_{H_{\gamma}}^{q}ds.\nonumber
\end{align}
As shown in \eqref{bound1} and \eqref{bound2}, we have 
\[\sup_{\e>0}\,\mathbb{E}\sup_{s\in[0,T]}\left(1+\vert u_{\epsilon}(s)\vert_{H_{\gamma}}^{q}+\vert\bar{u}(s)^{\vee}\vert_{H_{\gamma}}^{q}\right)\leq C\]
Thus, after taking supremum over time
and expectation in \eqref{cx41}, we obtain
\begin{equation}
\label{cx15}
\mathbb{E}\sup_{s\in[0,t]}\vert J_{\epsilon,1}(s)\vert_{H_{\gamma}}^{q}\leq C_{q,T}\tau+C_{q,T}\int_{\tau}^{t}\mathbb{E}\sup_{r\in[\tau,s]}\vert u_{\epsilon}(r)-\bar{u}(r)^{\vee}\vert_{H_{\gamma}}^{q}ds.
\end{equation}
For the second term $J_{\epsilon,2}(t)$, using again the linear growth of $B$ in $H_\gamma$, we have
\begin{align*}
\vert J_{\epsilon,2}(t)\vert_{H_{\gamma}}^{q} & \leq C_{q}\left|\int_{t-\tau}^{t}\left[S_{\epsilon}(t-s)-\bar{S}(t-s)^{\vee}\right]B(\bar{u}(s)^{\vee})ds\right|_{H_{\gamma}}^{q}\\
 & \quad+C_{q}\left|\int_{0}^{t-\tau}\left[S_{\epsilon}(t-s)-\bar{S}(t-s)^{\vee}\right]B(\bar{u}(s)^{\vee})ds\right|_{H_{\gamma}}^{q}\\
 & \leq C_{q,T}\tau\sup_{s\in[0,T]}\left(1+\vert\bar{u}(s)^{\vee}\vert_{H_{\gamma}}^{q}\right)\\
 & \quad+C_{q,T}\int_{0}^{T}\sup_{r\in[\tau,T]}\left|\left[S_{\epsilon}(r)-\bar{S}(r)^{\vee}\right]B(\bar{u}(s)^{\vee})\right|_{H_{\gamma}}^{q}ds.
\end{align*}
This implies
\[
\mathbb{E}\sup_{s\in[0,t]}\vert J_{\epsilon,2}(s)\vert_{H_{\gamma}}^{q}\leq C_{q,T}\tau+C_{q,T}\int_{0}^{T}\mathbb{E}\sup_{r\in[\tau,T]}\left|\left[S_{\epsilon}(r)-\bar{S}(r)^{\vee}\right]B(\bar{u}(s)^{\vee})\right|_{H_{\gamma}}^{q}ds.
\]
Together with \eqref{cx15}, we proved (\ref{eq: convergence theorme second term}) with
\[
H_{\e, 1}(\tau,T):=C_{q,T}\int_{0}^{T}\mathbb{E}\sup_{r\in[\tau,T]}\left|\left[S_{\epsilon}(r)-\bar{S}(r)^{\vee}\right]B(\bar{u}(s)^{\vee})\right|_{H_{\gamma}}^{q}ds.
\]
By (\ref{eq: semigroups weak convergence weighted space}) and (\ref{eq: bounded semigroups}),
using the dominated convergence theorem we have that 
\[
\lim_{\epsilon\rightarrow0}\int_{0}^{T}\mathbb{E}\sup_{r\in[\tau,T]}\left|\left[S_{\epsilon}(r)-\bar{S}(r)^{\vee}\right]B(\bar{u}(s)^{\vee})\right|_{H_{\gamma}}^{q}ds=0
\]
for any $0<\tau<T$ and this implies \eqref{cx12}.
\end{proof}
\begin{lem}
\label{lem: SPDE convergence third term} For every $q>2p$ and for every $0<\tau<T$, we have that
\begin{equation}
\mathbb{E}\sup_{s\in[0,t]}\vert I_{\epsilon,3}(s)\vert_{H_{\gamma}}^{q}\leq C_{q,T}\int_{\tau}^{t}\mathbb{E}\sup_{r\in[\tau,s]}\vert u_{\epsilon}(r)-\bar{u}(r)^{\vee}\vert_{H_{\gamma}}^{q}ds+H_{\e, 2}(T),\label{eq: convergence theorme third term}
\end{equation}
where $H_{\e, 2}(T)$ satisfies that 
\begin{equation}
\label{cx13}
\lim_{\e\to 0} H_{\e, 2}(T)=0.
\end{equation}
\end{lem}

\begin{proof}
We have
\begin{align*}
I_{\epsilon,3}(t) & =\int_{0}^{t}S_{\epsilon}(t-s)\left[\Sigma(u_{\epsilon}(s))-\Sigma(\bar{u}(s)^{\vee})\right]d\mathcal{W}(s)\\
 & \quad+\int_{0}^{t}\left[S_{\epsilon}(t-s)-\bar{S}(t-s)^{\vee}\right]\Sigma(\bar{u}(s)^{\vee})d\mathcal{W}(s)\\
 & =:J_{\epsilon,1}(t)+J_{\epsilon,2}(t).
\end{align*}
By the factorization formula, for every $\alpha \in\,(0,1)$ we have  
\begin{align*}
J_{\epsilon,1}(t) & =\frac{\sin\pi\alpha}{\pi}\int_{0}^{t}(t-s)^{\alpha-1}S_{\epsilon}(t-s)Y_{\alpha}(s)ds,
\end{align*}
where 
\begin{align*}
Y_{\alpha}(s) & =\int_{0}^{s}(s-r)^{-\alpha}S_{\epsilon}(s-r)\left[\Sigma(u_{\epsilon}(r))-\Sigma(\bar{u}(r)^{\vee})\right]d\mathcal{W}(r).
\end{align*}
Now, since $m \in\,L^p(\rtwo)$, for some $p \in\,(1,\infty)$, we can find $\alpha \in\,(0,1)$ and $q>1$ such that $2p<1/\alpha<q$. Then, using
Proposition \ref{prop: semigroup on the weight} and H\"older's inequality,
we have
\begin{align*}
\mathbb{E}\sup_{s\in[0,t]}\vert J_{\epsilon,1}(s)\vert_{H_{\gamma}}^{q} & \leq C_{q,T}\int_{0}^{t}\mathbb{E}\vert Y_{\alpha}(s)\vert_{H_{\gamma}}^{q}ds.
\end{align*}
By  Lemma \ref{lem: Semigroup property 1}, due to the Lipschitz continuity of $\sigma$,
we have that
\begin{align*}
\mathbb{E}\vert Y_{\alpha}(s)\vert_{H_{\gamma}}^{q} & \leq C_{q,T,\alpha}\,C^{q}\,\mathbb{E}\left(\int_{0}^{s}(s-r)^{-2\alpha}(s-r)^{-(p-1)/p}\vert u_{\epsilon}(r)-\bar{u}(r)^{\vee}\vert_{H_{\gamma}}^{2}dr\right)^{q/2}.
\end{align*}
Then, from Young's inequality and estimates \eqref{bound1} and \eqref{bound2}, we obtain
\begin{align*}
\int_{0}^{t}\mathbb{E}\vert Y_{\alpha}(s)\vert_{H_{\gamma}}^{q}ds & \leq C_{q,T}\left(\int_{0}^{t}s^{-2\alpha-(p-1)/p}ds\right)^{q/2} \mathbb{E} \int_{0}^{t}\vert u_{\epsilon}(s)-\bar{u}(s)^{\vee}\vert_{H_{\gamma}}^{q}ds\\
 & \leq C_{q,T}\mathbb{E}\int_{0}^{t}\vert u_{\epsilon}(s)-\bar{u}(s)^{\vee}\vert_{H_{\gamma}}^{q}ds\\
 & \leq C_{q,T}\tau\left(\mathbb{E}\sup_{s\in[0,T]}\vert u_{\epsilon}(s)\vert_{H_{\gamma}}^{q}+\mathbb{E}\sup_{s\in[0,T]}\vert\bar{u}(s)^{\vee}\vert_{H_{\gamma}}^{q}\right)\\
 & \quad+C_{q,T}\left(\int_{\tau}^{t}\mathbb{E}\sup_{r\in[\tau,s]}\vert u_{\epsilon}(r)-\bar{u}(r)^{\vee}\vert_{H_{\gamma}}^{q}ds\right)\\
 & \leq C_{q,T}\left(\tau+\int_{\tau}^{t}\mathbb{E}\sup_{r\in[\tau,s]}\vert u_{\epsilon}(r)-\bar{u}(r)^{\vee}\vert_{H_{\gamma}}^{q}ds\right).
\end{align*}
This implies 
\[
\mathbb{E}\sup_{s\in[0,t]}\vert J_{\epsilon,1}(s)\vert_{H_{\gamma}}^{q}\leq C_{q,T}\left(\tau+\int_{\tau}^{t}\mathbb{E}\sup_{r\in[\tau,s]}\vert u_{\epsilon}(r)-\bar{u}(r)^{\vee}\vert_{H_{\gamma}}^{q}ds\right).
\]
Again, using the factorization formula
\begin{align*}
J_{\epsilon,2}(t) & =\frac{\sin\pi\alpha}{\pi}\int_{0}^{t}(t-s)^{\alpha-1}S_{\epsilon}(t-s)Y_{\alpha,1}(s)ds\\
 & \quad+\frac{\sin\pi\alpha}{\pi}\int_{0}^{t}(t-s)^{\alpha-1}\left[S_{\epsilon}(t-s)-\bar{S}(t-s)^{\vee}\right]Y_{\alpha,2}(s)ds,
\end{align*}
where 
\[
Y_{\alpha,1}(s)=\int_{0}^{s}(s-r)^{-\alpha}\left[S_{\epsilon}(s-r)-\bar{S}(s-r)^{\vee}\right]G(\bar{u}(r)^{\vee})d\mathcal{W}(r),
\]
and
\[
Y_{\alpha,2}(s)=\int_{0}^{s}(s-r)^{-\alpha}\bar{S}(s-r)^{\vee}G(\bar{u}(r)^{\vee})d\mathcal{W}(r).
\]
Then 
\begin{align*}
\mathbb{E}\sup_{s\in[0,t]}\vert J_{\epsilon,2}(s)\vert_{H_{\gamma}}^{q} & \leq C_{T}\int_{0}^{t}\mathbb{E}\vert Y_{\alpha,1}(s)\vert_{H_{\gamma}}^{q}ds\\
 & \quad+C_{T}\mathbb{E}\sup_{s\in[0,t]}\int_{0}^{s}\left|\left[S_{\epsilon}(s-r)-\bar{S}(s-r)^{\vee}\right]Y_{\alpha,2}(r)\right|_{H_{\gamma}}^{q}dr\\
 & =:K_{\epsilon,1}(t)+K_{\epsilon,2}(t).
\end{align*}
Here 
\begin{align*}
\mathbb{E}\vert Y_{\alpha,1}(s)\vert_{H_{\gamma}}^{q} & \leq C^{q}\mathbb{E}\left(\int_{0}^{s}(s-r)^{-2\alpha}\sum_{j=1}^{\infty}\left|\left[S_{\epsilon}(s-r)-\bar{S}(s-r)^{\vee}\right]G(\bar{u}(r)^{\vee})e_{j}\right|_{H_{\gamma}}^{2}dr\right)^{q/2}.
\end{align*}
By Lemma \ref{lem: convergenc of the infinite sum} and (\ref{eq: bounded semigroups}),
using the dominated convergence theorem, we have that 
\[\lim_{\epsilon\rightarrow0}\mathbb{E}\vert Y_{\alpha,1}(s)\vert_{H_{\gamma}}^{q}=0.\]
This implies that 
\[
\lim_{\epsilon\rightarrow0}\sup_{t\in[0,T]}K_{\epsilon,1}(t)=0.
\]
For $K_{\epsilon,2}$, by (\ref{eq: bounded semigroups}) and 
the dominated convergence theorem, we again have 
\[
\lim_{\epsilon\rightarrow0}\sup_{t\in[0,T]}K_{\epsilon,2}(t)=0.
\]
Therefore, if we define 
\[H_{\e, 2}(T):=\sup_{t \in\,[0,T]} K_{\e, 1}(t)+\sup_{t \in\,[0,T]} K_{\e, 2}(t),\]
our proof is complete.
\end{proof}

\section{A weaker type convergence if dT/dz=0}
\label{section4}
In \cite{CerraiFreidlin2019}, it has been shown that if  Assumption \ref{Assumption 2} is verified, that is
\[
\frac{dT_{k}(z)}{dz}\neq0,\qquad(z,k)\in\Gamma,
\]
then for any $u \in\,H_\gamma$ and $0<\tau<T$
\begin{equation}
\label{cx20}
\lim_{\epsilon\rightarrow0}\sup_{t\in[\tau,T]}\left|\mathbf{E}_{x}u(X_{\epsilon}(t))-\bar{\mathbf{E}}_{\Pi(x)}u^{\wedge}(\bar{Y}(t))\right|=0.
\end{equation}

In \cite{CerraiFreidlin2019}, Assumption \ref{Assumption 2} is actually used  to say that, as shown in  \cite[Lemma 4.3]{Freidlin2002},  if $\alpha \in\,(4/7,2/3)$ then
for every $u \in\,C^2_b(\rtwo)$ and for every compact set $K \in\,\rtwo$
\begin{equation}
\label{cx21}
\lim_{\e\to 0} \sup_{x \in\,K}\left|\mathbf{E}_x u(X_\e(\e^\alpha))-(u^\wedge)^{\vee}(x)\right|=0.\end{equation}
When Assumption  \ref{Assumption 2} is not satisfied, we don't have a way to prove \eqref{cx21}, which is a key ingredient in the proof of \eqref{cx20}.
In this section,  we will
show that when Assumption \ref{Assumption 2} is not verified and hence we cannot prove  \eqref{cx21}, then limit \eqref{cx20} can be replaced by  the following weaker type of convergence.

\begin{thm}
\label{teo1}
Under Assumptions  \ref{Assumption 1}, \ref{Assumption 3} and \ref{Assumption 4}, for any $0\leq \tau<T$ and any
 compact set $K\subset\rtwo$, we have 
\begin{equation}
\lim_{\epsilon\rightarrow0}\sup_{x\in K}\left|\int_{\tau}^{T}\left[\mathbf{E}_{x}u(X_{\epsilon}(t))-\bar{\mathbf{E}}_{\Pi(x)}u^{\wedge}(\bar{Y}(t))\right]\theta(t)dt\right|=0\label{eq: weak convergence main result}
\end{equation}
for any $u\in C_{b}(\rtwo)$ and $\theta\in C_{b}([\tau,T])$.
\end{thm}
  
To prove Theorem \ref{teo1}, we need the following notations. For $\Pi(x)=(z,k)$ in the interior of edge $I_{k}$, we set
$T(x)=T_{k}(z)$. Given a compact set $K \subset \rtwo$ and  $\delta>0$,
we denote 
\[
T_{M, \delta}(K):=\sup_{x\in K\backslash G(\pm\delta)}T(x),\qquad T_{m, \delta}(K):=\inf_{x\in K\backslash G(\pm\delta)}T(x).
\]
Here we remove a small neighborhood of all the vertices,  $G(\pm\delta)$, when taking the supremum and infimum. Therefore we always have that $T_{M, \delta}(K)<\infty$ and $T_{m, \delta}(K)>0$ (see \cite[Chapter 8]{FreidlinWentzell1998}).

Now, suppose $(z,0)\in\Gamma$ is  such that
\begin{equation}
z\geq\max_{i=1,\cdots,m}H(x_{i})+1,\label{eq: cutoff point local condition}
\end{equation}
where $x_{1,}\cdots,x_{m}$ are the critical points of the Hamiltonian
$H$. We define the stopping time 
\[
\rho_{\epsilon,z}:=\inf\left\{ t\geq0:H(X_{\epsilon}(t))\geq z\right\} ,
\]
which is finite almost surely by Theorem \ref{thm: Kernels upper bound}. 
It is proved in \cite[Lemma 8.3.2]{FreidlinWentzell1998} that for any compact set $K\in \mathbb{R}^2$, there exists a $\e_0>0$ such that the family of distributions corresponding to the processes $\{\Pi(X_{\epsilon}(\cdot)) : \epsilon\in(0,\epsilon_{0}), X_\epsilon (0)\in K \}$
is tight in $C([0,T];\Gamma)$ for every $T>0$. This implies that for any given $\eta>0$ and $T>0$, there exists  $(z_{\eta},0)\in\Gamma$
satisfying (\ref{eq: cutoff point local condition}) such that
\[
\sup_{x\in K,0<\epsilon\leq\epsilon_{0}}\mathbf{P}_{x}\left\{ \sup_{t\leq T} H(X_\e(t))\geq z_\eta \right\} \leq\eta,
\]
which is equivalent to 
\begin{equation}\label{eq: cutoff out of a large set}
\sup_{x\in K,0<\epsilon\leq\epsilon_{0}}\mathbf{P}_{x}\left\{ \rho_{\epsilon,z_{\eta}}\leq T\right\} \leq\eta,
\end{equation}
i.e., the probability of  processes $X_\epsilon (t)$ hitting the level curve $C(z_\eta)$ before time $T$ are uniformly less than $\eta$ for any initial data $x\in K$ and $1<\epsilon<\epsilon_0$.
Given $\eta>0$ and  $z_{\eta}$ as  in \eqref{eq: cutoff out of a large set},
for any $0<\delta'<\delta$, define 
\begin{equation}
\sigma_{n}^{\epsilon,\eta,\delta,\delta'}:=\inf\left\{ t\geq\tau_{n}^{\epsilon,\eta,\delta,\delta'}:X_{\epsilon}(t)\in G(\pm\delta)^{c}\right\} ,\label{eq: time left neighborhood of vertices}
\end{equation}
and
\begin{equation}
\tau_{n}^{\epsilon,\eta,\delta,\delta'}:=\inf\left\{ t\geq\sigma_{n-1}^{\epsilon,\eta,\delta,\delta'}:X_{\epsilon}(t)\in D(\pm\delta')\cup C(z_{\eta})\right\} .\label{eq: time enter a smaller neighborhood of vertices}
\end{equation}
We set $\tau_{0}^{\epsilon,\eta,\delta,\delta'}=0$. After the process
$X_{\epsilon}(t)$ reaches $C(z_{\eta})$, all $\tau_{n}^{\epsilon,\eta,\delta,\delta'}$
and $\sigma_{n}^{\epsilon,\eta,\delta,\delta'}$ are taken equal $\rho_{\epsilon,z_{\eta}}$.

\subsection{A weaker type of convergence for the semigroup}
\label{WeakConvergSemigroup}
In order to prove Theorem \ref{teo1},
it is sufficient  to prove 
\begin{equation}\label{eq: teo equivalent}
\lim_{\epsilon\rightarrow0}\sup_{x\in K}\left|\int_{0}^{T}\left[\mathbf{E}_{x}u(X_{\epsilon}(t))-\bar{\mathbf{E}}_{\Pi(x)}u^{\wedge}(\bar{Y}(t))\right]\theta(t)dt\right|=0.
\end{equation}
Actually, if this is the case we can use $\int_{\tau}^{T}=\int_{0}^{T}-\int_{0}^{\tau}$ to
obtain \eqref{eq: weak convergence main result}.

Thanks to \cite[Lemma A.3]{CerraiFreidlin2019}, for any $0<\tau<T$ and
$x\in\rtwo$, we have 
\[
\lim_{\epsilon\rightarrow0}\sup_{t\in[\tau,T]}\left|\mathbf{E}_{x}(u^{\wedge})^{\vee}(X_{\epsilon}(t))-\bar{\mathbf{E}}_{\Pi(x)}u^{\wedge}(\bar{Y}(t))\right|=0.
\]
In fact, given a compact subset $K\subset \rtwo$, we further have that
\begin{equation}
\lim_{\epsilon\rightarrow0}\sup_{x\in K,t\in[\tau,T]}\left|\mathbf{E}_{x}(u^{\wedge})^{\vee}(X_{\epsilon}(t))-\bar{\mathbf{E}}_{\Pi(x)}u^{\wedge}(\bar{Y}(t))\right|=0.\label{eq: uniform convergence local equation}
\end{equation}
Notice that we have the decomposition below
\begin{align*}
\int_{0}^{T}\left[\mathbf{E}_{x}u(X_{\epsilon}(t))-\bar{\mathbf{E}}_{\Pi(x)}u^{\wedge}(\bar{Y}(t))\right]\theta(t)dt & =\int_{0}^{T}\left[\mathbf{E}_{x}u(X_{\epsilon}(t))-\mathbf{E}_{x}(u^{\wedge})^{\vee}(X_{\epsilon}(t))\right]\theta(t)dt\\
 & \quad+\int_{0}^{\tau}\left[\mathbf{E}_{x}(u^{\wedge})^{\vee}(X_{\epsilon}(t))-\bar{\mathbf{E}}_{\Pi(x)}u^{\wedge}(\bar{Y}(t))\right]\theta(t)dt\\
 & \quad+\int_{\tau}^{T}\left[\mathbf{E}_{x}(u^{\wedge})^{\vee}(X_{\epsilon}(t))-\bar{\mathbf{E}}_{\Pi(x)}u^{\wedge}(\bar{Y}(t))\right]\theta(t)dt\\
 & =:I_{1}+I_{2}+I_{3}.
\end{align*}
Since $u$, $(u^{\wedge})^{\vee}$ and $\varphi$ are bounded, we
can choose $\tau$ small enough to control $I_{2}$. Then using (\ref{eq: uniform convergence local equation})
we can control $I_{3}$. Hence, in order to obtain \eqref{eq: weak convergence main result},  it is enough to prove the following result.

\begin{lem}
Suppose $K$ is a compact subset of $\rtwo$. Then, for every  $T>0$ and $\theta\in C([0,T])$ it holds that
\begin{equation}
\label{cx43}
\lim_{\e\to 0} \sup_{x\in K}\left|\int_{0}^{T}\left[\mathbf{E}_{x}u(X_{\epsilon}(t))-\mathbf{E}_{x}(u^{\wedge})^{\vee}(X_{\epsilon}(t))\right]\theta(t)dt\right|=0.\end{equation}
\end{lem}

\begin{proof}
Actually we can assume that $u\in C_{b}^{1}(\rtwo)$ (and $\theta\in C^{1}([0,T])$),
because for any $u\in C_{b}(\rtwo)$ (and $\theta\in C([0,T])$)
we can find an approximation sequence $\left\{ u_{n}\right\} \subset C_{b}^{1}(\rtwo)$
(and $\left\{ \theta_{n}\right\} \subset C^{1}([0,T])$) such
that $u_{n}\rightarrow u$ in $L^{\infty}(\rtwo)$ (and $\theta_{n}\rightarrow\theta$
in $L^{\infty}([0,T]$). 

Since $u\in C_{b}^{1}(\rtwo)$ and $\theta\in C^{1}([0,T])$,
we can define
\[M_1:=\Vert u-(u^{\wedge})^{\vee}\Vert_{L^{\infty}},\ \ \ 
M_2:=\max\left\{ \Vert\theta\Vert_{L^{\infty}},\Vert\theta'\Vert_{L^{\infty}}\right\},\ \ \ M_3:=\Vert\nabla u\Vert_{L^{\infty}}.\]
 By \eqref{eq: cutoff out of a large set}, for every $\eta>0$
we can choose $z_{\eta}$ large enough such that $\sup_{x\in K}\mathbf{P}_{x}(\rho_{\epsilon,z_{\eta}}\leq T)\leq\eta$. Hence,
\begin{align*}
\mathbf{E}_{x}\int_{0}^{T}\left[u(X_{\epsilon}(t))-(u^{\wedge})^{\vee}(X_{\epsilon}(t))\right]\theta(t)dt & =\mathbf{E}_{x}\int_{0}^{T\wedge\rho_{\epsilon,z_{\eta}}}\left[u(X_{\epsilon}(t))-(u^{\wedge})^{\vee}(X_{\epsilon}(t))\right]\theta(t)dt\\
 & \quad+\mathbf{E}_{x}\int_{T\wedge\rho_{\epsilon,z_{\eta}}}^{T}\left[u(X_{\epsilon}(t))-(u^{\wedge})^{\vee}(X_{\epsilon}(t))\right]\theta(t)dt.
\end{align*}
For the second term on the right hand side, 
\[
\sup_{x\in K}\left|\mathbf{E}_{x}\int_{T\wedge\rho_{\epsilon,z_{\eta}}}^{T}\left[u(X_{\epsilon}(t))-(u^{\wedge})^{\vee}(X_{\epsilon}(t))\right]\theta(t)dt\right|\leq TM_{1}M_{2}\mathbf{P}_{x}(\rho_{\epsilon,z_{\eta}}\leq T)\leq TM_{1}M_{2}\eta.
\]
Now we fixed $z_{\eta}$ and the stoping time $\rho_{\epsilon,z_{\eta}}$.
For the stopping times $\tau_{i}^{\epsilon,\eta,\delta,\delta/2}$
and $\sigma_{i}^{\epsilon,\eta,\delta,\delta/2}$ defined in (\ref{eq: time left neighborhood of vertices})
and (\ref{eq: time enter a smaller neighborhood of vertices}) with
$\delta'=\frac{\delta}{2}$, we set $\tau_{i}=\tau_{i}^{\epsilon,\eta,\delta,\delta/2}\wedge T\wedge\rho_{\epsilon,z_{\eta}}$
and $\sigma_{i}=\sigma_{i}^{\epsilon,\eta,\delta,\delta/2}\wedge T\wedge\rho_{\epsilon,z_{\eta}}$.
Then, recalling that $\tau_0=0$, we have  
\begin{align*}\mathbf
\mathbf{E}_{x}\int_{0}^{T\wedge\rho_{\epsilon,z_{\eta}}}\left[u(X_{\epsilon}(t))-(u^{\wedge})^{\vee}(X_{\epsilon}(t))\right]\theta(t)dt & =\sum_{i=0}^{\infty}\mathbf{E}_{x}\int_{\tau_{i}}^{\sigma_{i}}\left[u(X_{\epsilon}(t))-(u^{\wedge})^{\vee}(X_{\epsilon}(t))\right]\theta(t)dt\\
 & \quad+\sum_{i=0}^{\infty}\mathbf{E}_{x}\int_{\sigma_{i}}^{\tau_{i+1}}\left[u(X_{\epsilon}(t))-(u^{\wedge})^{\vee}(X_{\epsilon}(t))\right]\theta(t)dt.
\end{align*}
For the first term on the right hand side, we have
\begin{align*}
 & \quad\left|\sum_{i=0}^{\infty}\mathbf{E}_{x}\int_{\tau_{i}}^{\sigma_{i}}\left[u(X_{\epsilon}(t))-(u^{\wedge})^{\vee}(X_{\epsilon}(t))\right]\theta(t)dt\right|\\
 & \leq M_{1}M_{2}\sum_{i=0}^{\infty}\mathbf{E}_{x}\left[\sigma_{i}-\tau_{i}\right]\\
 & \leq M_{1}M_{2}\sum_{i=0}^{\infty}\mathbf{P}_x(\tau_{i}^{\epsilon,\eta,\delta,\delta/2}<T)\left[\sup_{y\in D(\pm\delta/2)}\mathbf{E}_{y}\sigma_{0}^{\epsilon,\eta,\delta,\delta/2}\right]\\
 & \leq M_{1}M_{2}\left[\sup_{y\in D(\pm\delta/2)}\mathbf{E}_{y}\sigma_{0}^{\epsilon,\eta,\delta,\delta/2}\right]e^{T}\sum_{i=0}^{\infty}\mathbf{E}_{x}e^{-\tau_{i}^{\epsilon,\eta,\delta,\delta/2}}.
\end{align*}
Recall that from \cite[(8.3.14)]{FreidlinWentzell1998} and \cite[A.19]{CerraiFreidlin2019}
we have
\[
\sup_{x\in K}\sum_{i=0}^{\infty}\mathbf{E}_{x}e^{-\tau_{i}^{\epsilon,\eta,\delta,\delta/2}}\leq\frac{C}{\delta},
\]
and from \cite[(8.5.17)]{FreidlinWentzell1998} and \cite[A.21]{CerraiFreidlin2019}
we have
\[
\sup_{y\in D(\pm\delta/2)}\mathbf{E}_{y}\sigma_{0}^{\epsilon,\eta,\delta,\delta/2}\leq C\delta^{2}\vert\log\delta\vert.
\]
Since there exists  $\delta_{0}>0$ such that $C\delta\vert\log\delta\vert<\eta$
for all $0<\delta<\delta_{0}$, from the inequality above we get
\begin{equation}
\label{cx45}
\sup_{x\in K}\left|\sum_{i=0}^{\infty}\mathbf{E}_{x}\int_{\tau_{i}}^{\sigma_{i}}\left[u(X_{\epsilon}(t))-(u^{\wedge})^{\vee}(X_{\epsilon}(t))\right]\theta(t)dt\right|\leq CM_{1}M_{2}e^{T}\delta\vert\log\delta\vert<M_{1}M_{2}e^{T}\eta.
\end{equation}
Using Lemma \ref{lem: estimate on each edge} below we have 
\begin{align*}
 & \quad\sup_{x\in K}\left|\sum_{i=0}^{\infty}\mathbf{E}_{x}\int_{\sigma_{i}}^{\tau_{i+1}}\left[u(X_{\epsilon}(t))-(u^{\wedge})^{\vee}(X_{\epsilon}(t))\right]\theta(t)dt\right|\\
 & \leq\sum_{i=0}^{\infty}\mathbf{P}_x(\sigma_{i}^{\epsilon,\eta,\delta,\delta/2}\leq T)\sup_{x\in K}\left|\mathbf{E}_{x}\int_{\sigma_{i}}^{\tau_{i+1}}\left[u(X_{\epsilon}(t))-(u^{\wedge})^{\vee}(X_{\epsilon}(t))\right]\theta(t)dt\right|\\
 & \leq C\,\sum_{i=0}^{\infty}\mathbf{P}_x(\sigma_{i}^{\epsilon,\eta,\delta,\delta/2}\leq T)\sqrt{\epsilon}\\
 & \leq C\left( 1+\sum_{i=0}^{\infty}\mathbf{P}_x(\tau_{i}^{\epsilon,\eta,\delta,\delta/2}\leq T)\right) \sqrt{\epsilon}\\
 & \leq C\,e^{T}\left(\frac{C}{\delta}+1\right)\sqrt{\epsilon}.
\end{align*}
This implies that we can find $\epsilon_{0}>0$ small enough such that for all $0<\epsilon\leq\epsilon_{0}$
\[
\sup_{x\in K}\left|\sum_{i=0}^{\infty}\mathbf{E}_{x}\int_{\sigma_{i}}^{\tau_{i+1}}\left[u(X_{\epsilon}(t))-(u^{\wedge})^{\vee}(X_{\epsilon}(t))\right]\theta(t)dt\right|\leq\eta.
\]
This, together with \eqref{cx45} gives \eqref{cx43}.
\end{proof}

\begin{lem}
\label{lem: estimate on each edge}
For every given $\delta>0$ and each $i\in\mathbf{N}$, we
have  
\[
\sup_{x\in K}\left|\mathbf{E}_{x}\int_{\sigma_{i}}^{\tau_{i+1}}\left[u(X_{\epsilon}(t))-(u^{\wedge})^{\vee}(X_{\epsilon}(t))\right]\theta(t)dt\right|\leq C\,\sqrt{\epsilon}.
\]
where the constant $C$ depends on $M_{j}$ with $j=1,2,3$,  $T$, $T_{M,
\delta/2}:=T_{M,
\delta/2}(\overline{G(0,z_{\eta})})$,
and $T_{m, \delta/2}:=T_{m, \delta/2}(\overline{G(0,z_{\eta})})$ and 
\[M_{4,\delta}:=\sup_{x \in\,G(0,z_{\eta})\backslash G(\pm\delta/2)}|\nabla((u^{\wedge})^{\vee})(x)|<\infty.\]

\end{lem}

\begin{proof}
We introduce the following sequence of stoping times 
$\sigma_{i}=s_{0}\leq s_{1}\leq s_{2}\leq\cdots\leq s_{\nu}=\tau_{i+1},$
by setting 
\[s_{k+1}=\left[s_{k}+\epsilon T(X_{\epsilon}(s_{k}))\right]\wedge\tau_{i+1},\ \ \ \ k=1,\ldots,\nu-1.\]
Notice that we must have $\nu\leq N:=[\e^{-1} T/T_{m, \delta/2}]+1$, since $\vert\tau_{i+1}-\sigma_{i}\vert\leq T$.
Then we have
\[
\left|\mathbf{E}_{x}\int_{\sigma_{i}}^{\tau_{i+1}}\left[u(X_{\epsilon}(t))-(u^{\wedge})^{\vee}(X_{\epsilon}(t))\right]\theta(t)dt\right|\leq\sum_{k=0}^{N-1}\left|\mathbf{E}_{x}\int_{s_{k}}^{s_{k+1}}\left[u(X_{\epsilon}(t))-(u^{\wedge})^{\vee}(X_{\epsilon}(t))\right]\theta(t)dt\right|.
\]
For each $k$, we have that
\begin{align*}
 & \quad\left|\mathbf{E}_{x}\int_{s_{k}}^{s_{k+1}}\left[u(X_{\epsilon}(t))-(u^{\wedge})^{\vee}(X_{\epsilon}(t))\right]\theta(t)dt\right|\\
 & \leq\left|\mathbf{E}_{x}\int_{s_{k}}^{s_{k}+\epsilon T(X_{\epsilon}(s_{k}))}\left[u(X_{\epsilon}(t))-(u^{\wedge})^{\vee}(X_{\epsilon}(t))\right]\theta(t)dt\cdot1_{\{s_{k}+\epsilon T(X_{\epsilon}(s_{k}))<\tau_{i+1}\}}\right|\\
 & \quad+\left|\mathbf{E}_{x}\int_{s_{k}}^{\tau_{i+1}}\left[u(X_{\epsilon}(t))-(u^{\wedge})^{\vee}(X_{\epsilon}(t))\right]\theta(t)dt\cdot1_{\{s_{k}+\epsilon T(X_{\epsilon}(s_{k}))\geq\tau_{i+1},s_{k}<\tau_{i+1}\}}\right|\\
 & =:I_{1}+I_{2}.
\end{align*}
By the definition of $\sigma_{i}$ and $\tau_{i+1}$, $X_{\epsilon}(s_{k})\in G(0,z_{\eta})\backslash G(\pm\delta/2)$.
For $I_{2}$, we have 
\[\vert\tau_{i+1}-s_{k}\vert\leq\epsilon T(X_{\epsilon}(s_{k}))\leq\epsilon T_{M, \delta/2}(G(0,z_{\eta})),\]
which implies that
\[
I_{2}\leq\mathbf{P}_x\{s_{k}+\epsilon T(X_{\epsilon}(s_{k}))\geq\tau_{i+1},s_{k}<\tau_{i+1}\}M_{1}M_{2}T_{M,\delta/2}\epsilon.
\]
For $I_{1}$, we use the decomposition  
\begin{align*}
u(X_{\epsilon}(t))-(u^{\wedge})^{\vee}(X_{\epsilon}(t)) & =\left[u(X_{\epsilon}(t))-u(x_{\epsilon}(t))+(u^{\wedge})^{\vee}(x_{\epsilon}(t))-(u^{\wedge})^{\vee}(X_{\epsilon}(t))\right]\\
 & \quad+\left[u(x_{\epsilon}(t))-(u^{\wedge})^{\vee}(x_{\epsilon}(t))\right]\\
 & =:U_{1}(t)+U_{2}(t),
\end{align*}
where $x_{\epsilon}(t)$ is the deterministic fast motion defined
by (\ref{eq: deterministic fast motion}), with initial condition $x_{\epsilon}(s_{k})=X_{\epsilon}(s_{k})$.
Then 
\begin{align*}
I_{1} & \leq\left|\mathbf{E}_{x}\int_{s_{k}}^{s_{k}+\epsilon T(X_{\epsilon}(s_{k}))}U_{1}(t)\theta(t)dt\cdot1_{\{s_{k}+\epsilon T(X_{\epsilon}(s_{k}))<\tau_{i+1}\}}\right|\\
 & \quad+\left|\mathbf{E}_{x}\int_{s_{k}}^{s_{k}+\epsilon T(X_{\epsilon}(s_{k}))}U_{2}(t)\theta(t)dt\cdot1_{\{s_{k}+\epsilon T(X_{\epsilon}(s_{k}))<\tau_{i+1}\}}\right|\\
 & =:I_{11}+I_{12},
\end{align*}
Since $u=(u^\wedge)^\vee$ on the level set $C_{H(x)}$ and $x_\e(t)$ moves on the same connected components of $C_{H(x)}$, for all $t \in\,[0,\e T(x)]$, 
\[
\int_{0}^{\epsilon T(x)}\left[u(x_{\epsilon}(t))-(u^{\wedge})^{\vee}(x_\e(t))\right]dt=0.
\]
Therefore, we have that
\begin{align*}
I_{12} & \leq\left|\mathbf{E}_{x}\int_{s_{k}}^{s_{k}+\epsilon T(X_{\epsilon}(s_{k}))}U_{2}(t)\left(\theta(t)-\theta(s_{k})\right)dt\cdot1_{\{s_{k}+\epsilon T(X_{\epsilon}(s_{k}))<\tau_{i+1}\}}\right|\\
 & \leq M_{1}M_{2}T_{M, \delta/2}^{2}\epsilon^{2}.
\end{align*}
Since processes $X_{\epsilon}(t)$
and $x_{\epsilon}(t)$ always stay in the region $G(0,z_{\eta})\backslash G(\pm\delta/2)$,
we have 
\[
I_{11}\leq\left|\mathbf{E}_{x}\int_{s_{k}}^{s_{k}+\epsilon T(X_{\epsilon}(s_{k}))}\left(M_3+M_{4, \delta}\right)\vert X_{\epsilon}(t)-x_{\epsilon}(t)\vert M_{2}dt\cdot1_{\{s_{k}+\epsilon T(X_{\epsilon}(s_{k}))<\tau_{i+1}\}}\right|.
\]
It is not difficult to check that
\[
\sup_{x\in G(0,z_{\eta})\backslash G(\pm\delta/2)}\mathbf{E}_{x}\vert X_{\epsilon}(\epsilon s\wedge\tau_{1}^{\epsilon,\eta,\delta,\delta/2})-x_{\epsilon}(\epsilon s\wedge\tau_{1}^{\epsilon,\delta,\delta/2})\vert\leq C(\epsilon s)^{\frac{1}{2}}
\]
for any $s\in[0,T(x)]$.
Then, by the Strong Markov property of the diffusion $X_{\epsilon}(t)$
we have that 
\begin{align*}
I_{11} & \leq \mathbf{P}_x\left( s_{k}+\epsilon T(X_{\epsilon}(s_{k}))<\tau_{i+1}\right) \sup_{x\in G(0,z_{\eta})\backslash G(\pm\delta/2)}\mathbf{E}_{x}\int_{0}^{\epsilon T(x)}(M_{3}+M_{4, \delta})M_{2}\vert X_{\epsilon}(t)-x_{\epsilon}(t)\vert dt \\
 & \leq(M_{3}+M_{4, \delta})M_{2}CT_{M, \delta/2}^{\frac{3}{2}}\epsilon^{\frac{3}{2}}.
\end{align*}
Notice that 
\[\sum_{k=0}^{N-1}\mathbf{P}\{s_{k}+\epsilon T(X_{\epsilon}(s_{k}))\geq\tau_{i+1},s_{k}<\tau_{i+1}\}=1.\]
Now we have 
\begin{align*}
 & \quad\left|\mathbf{E}_{x}\int_{\sigma_{i}}^{\tau_{i+1}}\left[u(X_{\epsilon}(t))-(u^{\wedge})^{\vee}(X_{\epsilon}(t))\right]\theta(t)dt\right|\\
 & \leq\sum_{k=0}^{N-1}\left|\mathbf{E}_{x}\int_{s_{k}}^{s_{k+1}}\left[u(X_{\epsilon}(t))-(u^{\wedge})^{\vee}(X_{\epsilon}(t))\right]\theta(t)dt\right|\\
 & \leq\sum_{k=0}^{N-1}\mathbf{P}\{s_{k}+\epsilon T(X_{\epsilon}(s_{k}))\geq\tau_{i+1},s_{k}<\tau_{i+1}\}M_{1}M_{2}T_{M, \delta/2}\,\epsilon\\
 & \quad+N\left[M_{1}M_{2}T_{M, \delta/2}^{2}\epsilon^{2}+(M_{3}+M_{4, \delta})M_{2}CT_{M, \delta/2}^{\frac{3}{2}}\epsilon^{\frac{3}{2}}\right]\\
 & \leq M_{1}M_{2}T_{M, \delta/2}\epsilon+TM_{1}M_{2}T_{M, \delta/2}\epsilon+T(M_{3}+M_{4, \delta})M_{2}CT_{M, \delta/2}^{\frac{1}{2}}\epsilon^{\frac{1}{2}}
\end{align*}
Finally,  as all of the estimates for $I_{11}$, $I_{12}$
and $I_{2}$ are uniform for initial data $x\in K$, our proof
is complete.
\end{proof}

\subsection{The corresponding weaker convergence of the SPDEs}
Now we consider the convergence of the SPDEs based on the convergence of the semigroups obtained in section \ref{WeakConvergSemigroup} without Assumption \ref{Assumption 2}. Notice that in
equation (\ref{SPDE}), the nonlinear functions $b$ and $\sigma$
are assumed to be Lipschitz and hence preserve the strong convergence in
$H_{\gamma}$. In this section, the semigroups converge in a weak
sense, and the nonlinear functions
no longer preserve it. This indicates that we can not obtain the same
convergence result we obtained earlier.
Here, we consider the special case  when $b=0$ and the noise is additive, i.e.
\begin{equation}
\left\{\begin{array}{l}
\partial_{t}u_{\epsilon}(t,x)  =\frac{1}{2}\Delta u_{\epsilon}(t,x)+\frac{1}{\epsilon}\langle\nabla^{\perp}H(x),\nabla u_{\epsilon}(t,x)\rangle+\partial_{t}\mathcal{W}(t,x),\label{eq: linear SPDE}\\
u_{\epsilon}(0,x)  =\varphi(x),\qquad x\in\mathbb{R}^{2},
\end{array}\right.
\end{equation}
and
\begin{equation}
\left\{\begin{array}{l}
\partial_{t}\bar{u}(t,z,k)  =\bar{L}\bar{u}(t,z,k)+\partial_{t}\bar{\mathcal{W}}(t,z,k),\label{eq: limit linear SPDE}\\
\bar{u}(0,z,k)  =\varphi^{\wedge}(z,k),\qquad(z,k)\in\Gamma. 
\end{array}\right.
\end{equation}

Similar to (\ref{eq: mild solution}),  mild solutions to (\ref{eq: linear SPDE}) and (\ref{eq: limit linear SPDE})
are defined to be
\[
u_{\epsilon}(t)=S_{\epsilon}(t)\varphi+\int_{0}^{t}S_{\epsilon}(t-s)d\mathcal{W}(s),
\]
and 
\[
\bar{u}(t)=\bar{S}(t)\varphi^\wedge+\int_{0}^{t}\bar{S}(t-s)d\bar{\mathcal{W}}(s).
\]

In what follows, we define operators
\[
R_{\epsilon}^{\theta}(\tau,T)=\int_{\tau}^{T}S_{\epsilon}(t)\theta(t)dt,
\]
and
\[
\bar{R}^{\theta}(\tau,T)=\int_{\tau}^{T}\bar{S}(t)^\vee\theta(t)dt.
\]
Then (\ref{eq: weak convergence main result}) is equivalent to 
\[
\lim_{\epsilon\rightarrow0}\sup_{x\in K}\left|R_{\epsilon}^{\theta}(\tau,T)u(x)-\bar{R}^{\theta}(\tau,T)u(x)\right|=0.
\]
Recall that $u$ is assumed to be in $C_{b}(\rtwo)$ in Theorem \ref{teo1}.
In the following proposition, we will extend it for $u\in H_{\gamma}$.
\begin{prop}
\label{prop: integral of semigroup converges weighted space}
Under Assumptions \ref{Assumption 1} and  \ref{Assumption 3},
we have
\[
\lim_{\epsilon\rightarrow0}\left|\left[R_{\epsilon}^{\theta}(\tau,T)-\bar{R}^{\theta}(\tau,T)\right]u\right|_{H_{\gamma}}=0
\]
for any $u\in H_{\gamma}(\rtwo)$, $0\leq \tau<T$  and $\theta\in C_{b}([\tau,T])$.
\end{prop}

\begin{proof}
Since $\left[R_{\epsilon}^{\theta}(\tau,T)-\bar{R}^{\theta}(\tau,T)\right]u$
converges point-wise for every $ u \in\,C_b(\rtwo)$, due to the dominated convergence theorem, we have
\[\lim_{\epsilon\rightarrow0}\left|\left[R_{\epsilon}^{\theta}(\tau,T)-\bar{R}^{\theta}(\tau,T)\right]u\right|_{H_{\gamma}}=0.\]
The function $\gamma$ introduced in Proposition \ref{prop: semigroup on the weight} satisfies that $\inf_{x \in\,K}\gamma^\vee(x)>0$
for every compact set $K \subset \rtwo$.
Then, using a localization argument, it is possible to prove that for any $u\in H_{\gamma}$ there exists a sequence  $\left\{ u_{n}\right\} _{n\geq1}\subset C_{b}(\rtwo)$
such that $u_{n}\rightarrow u$ in $H_{\gamma}$.
Thanks to (\ref{eq: bounded semigroups}) and \eqref{eq: limit kernel applied to the weight}, we have
\[\left|\left[R_{\epsilon}^{\theta}(\tau,T)-\bar{R}^{\theta}(\tau,T)\right](u-u_{n})\right|_{H_{\gamma}}\leq C_T\,|u-u_n|_{H_\gamma}.\]
This implies that
\begin{align*}
\left|\left[R_{\epsilon}^{\theta}(\tau,T)-\bar{R}^{\theta}(\tau,T)\right]u\right|_{H_{\gamma}} & \leq\left|\left[R_{\epsilon}^{\theta}(\tau,T)-\bar{R}^{\theta}(\tau,T)\right]u_{n}\right|_{H_{\gamma}}+C_T\,|u-u_n|_{H_\gamma},
\end{align*}
and the proof is done.
\end{proof}
Next we will show that the mild solutions $u_{\epsilon}$ to the SPDEs
(\ref{eq: linear SPDE}) converges to the mild solution $\bar{u}$ to \eqref{eq: limit linear SPDE}, for which we need the following lemma. 
\begin{lem}
\label{lem: Weak type convergence of SPDE}Under Assumptions \ref{Assumption 1} and \ref{Assumption 3}, 
for any fixed  $T>0$ and $\theta \in\,C([0,T])$ we have 
\begin{equation}
\label{cx57}
\lim_{\epsilon\rightarrow0}\sum_{j=1}^{\infty}\left|\int_{0}^{t}(S_{\epsilon}(s)-\bar{S}(s)^{\vee})e_{j}\theta(s)ds\right|_{H_{\gamma}}^{2}=0.
\end{equation}
\end{lem}

\begin{proof}
We will prove that for any $\delta>0$, there exists $\epsilon_{\delta}>0$
such that for any $0<\epsilon<\epsilon_{\delta}$
\[
\sum_{j=1}^{\infty}\left|\int_{0}^{t}(S_{\epsilon}(s)-\bar{S}(s)^{\vee})e_j\theta(s)ds\right|_{H_{\gamma}}^{2}\leq \delta.
\]
The spectral
measure $m\in L^{p}(\rtwo)$ for $p\in(1,\infty)$, which means that
$m^{1/2}\in L^{2p}(\mathbb{R}^2)$. Given $\eta>0$, we write $m=m_{1}+m_{2}$,
\[m_{1}:=m1_{\{m<\eta^{2}\}},\ \ \ \ \ m_{2}:=m1_{\{m\geq\eta^{2}\}}.\]
Then $m^{1/2}=m_{1}^{1/2}+m_{2}^{1/2}$ and
\begin{align*}
\sum_{j=1}^{\infty}\left|\int_{0}^{t}(S_{\epsilon}(s)-\bar{S}(s)^{\vee})e_j\theta(s)ds\right|_{H_{\gamma}}^{2} &  \leq 2\sum_{j=1}^{\infty}\left|\int_{0}^{t}(S_{\epsilon}(s)-\bar{S}(s)^{\vee})\left(\widehat{v_{j}m_{1}^{1/2}}\right)\theta(s)\,ds\right|_{H_{\gamma}}^{2}\\
  & \quad+2 \sum_{j=1}^{\infty}\left|\int_{0}^{t}(S_{\epsilon}(s)-\bar{S}(s)^{\vee})\left(\widehat{v_{j}m_{2}^{1/2}}\right)
  \theta(s)\,ds\right|_{H_{\gamma}}^{2}\\
  &=:I^\eta_{1,\epsilon}+I^\eta_{2,\epsilon}.
\end{align*}
For the first term, since $\Vert m_{1}\Vert_{L^{2p}}\leq\eta\Vert m\Vert_{L^{p}}^{1/2}$, due to Lemma \ref{lem: Semigroup property 1}
we have 
\begin{align*}
I^\eta_{1,\epsilon}  &  \leq 2 \sum_{j=1}^{\infty}\left|\int_{0}^{t}(S_{\epsilon}(s)-\bar{S}(s)^{\vee})\left(\widehat{v_{j}m_{1}^{1/2}}\right)\theta(s)\right|_{H_{\gamma}}^{2}ds \\
& \leq C_T\int_{0}^{t}\Vert m_{1}\Vert_{L^{2p}}s^{-(2p-1)/2p}\,ds\|\theta\|_{L^\infty}^2\\
 & \leq C_T\eta\Vert m\Vert_{L^{p}}^{1/2}\int_{0}^{T}s^{-(2p-1)/2p}ds \|\theta\|_{L^\infty}^2.
\end{align*}
Hence we can choose $\eta_\delta>0$ small enough such that
\begin{equation}
\label{cx55}
\sup_{\e>0}\,I^{\eta_\delta}_{1,\epsilon} <\frac \delta 3.
\end{equation}

For the second term, for every $N \in\,\mathbb{N}$ we have
\begin{align*}
I^\eta_{2,\epsilon}  &  = 2 \sum_{j=1}^{\infty}\left|\int_{0}^{t}(S_{\epsilon}(s)-\bar{S}(s)^{\vee})\left(\widehat{v_{j}m_{2}^{1/2}}\right)
  \theta(s)\,ds\right|_{H_{\gamma}}^{2}\\
  &\leq C  \sum_{j=1}^{N}\left|\int_{0}^{t}(S_{\epsilon}(s)-\bar{S}(s)^{\vee})\left(\widehat{v_{j}m_{2}^{1/2}}\right)
  \theta(s)\,ds\right|_{H_{\gamma}}^{2}\\
  & \quad +C_T \sum_{j=N+1}^{\infty}\int_0^t\left|(S_{\epsilon}(s)-\bar{S}(s)^{\vee})\left(\widehat{v_{j}m_{2}^{1/2}}\right)\right|_{H_\gamma}^2|\theta(s)|^2\,ds\\
  &=:J^{N,\eta}_{1,\epsilon}+J^{N,\eta}_{2,\epsilon}.
\end{align*}
Since 
\[\Vert m_2 \Vert_{L^1}\leq \eta^{-(2p-2)}\|m\|_{L^p}^p,\] by proceeding as in the proof of Lemma \ref{lem: convergenc of the infinite sum}, once fixed $\delta>0$ there exists $N_{\delta} \in\,\mathbb{N}$ such that
\begin{equation}
\label{cx56}
\sup_{\e>0} J^{N_\delta,\eta_\delta}_{2,\epsilon}<\frac \delta 3.
\end{equation}
Then, once fixed $N_\delta$, due to Proposition \ref{prop: integral of semigroup converges weighted space} we have that there exists $\epsilon_\delta>0$ such that
\[J^{N_\delta,\eta_\delta}_{1,\epsilon}<\frac \delta 3,\ \ \ \ \mbox{for }\epsilon<\epsilon_\delta.\]
This inequality, together with \eqref{cx56} and \eqref{cx55}, implies \eqref{cx57}.

\end{proof}
\begin{thm}
\label{thmLinear}
Suppose the Hamiltonian $H$ satisfies Assumption \ref{Assumption 1}.
The spectral measure to the spatially homogeneous Wiener process $\mathcal{W}(t)$
satisfies Assumption \ref{Assumption 3}. Let $u_{\epsilon}\in\mathcal{H}_{q}$
be the unique mild solutions to (\ref{eq: linear SPDE}) and $\bar{u}\in\bar{\mathcal{H}}_{q}$
be the unique mild solution to (\ref{eq: limit linear SPDE}) with
the same initial condition $\varphi$ and $\varphi^{\wedge}$, respectively. Then for any fixed $T>0$, $q\geq 1$ and $\theta \in\,C([0,T])$, we have
that 
\begin{equation}
\label{cx30}
\lim_{\epsilon\rightarrow0}\mathbb{E}\left|\int_{0}^{T}\left[u_{\epsilon}(t)-\bar{u}(t)^{\vee}\right]\theta(t)dt\right|_{H_{\gamma}}^{q}=\lim_{\epsilon\rightarrow0}\mathbb{E}\left|\int_{0}^{T}\left[u_{\epsilon}(t)^{\wedge}-\bar{u}(t)\right]\theta(t)dt\right|_{\bar{H}_{\gamma}}^{q}=0.
\end{equation}
\end{thm}

\begin{proof}
We have
\begin{align*}
\int_{0}^{T}\left[u_{\epsilon}(t)-\bar{u}(t)^{\vee}\right]\theta(t)dt & =\int_{0}^{T}\left[S_{\epsilon}(t)\varphi-\bar{S}(t)^{\vee}\varphi\right]\theta(t)dt\\
 & \quad+\int_{0}^{T}\int_{0}^{t}\left[S_{\epsilon}(t-s)-\bar{S}(t-s)^{\vee}\right]d\mathcal{W}(s)\theta(t)dt\\
 & =:I_{\epsilon,1}+I_{\epsilon,2}.
\end{align*}
By Proposition \ref{prop: integral of semigroup converges weighted space}
\begin{equation}
\label{cx31}
\lim_{\epsilon\rightarrow0}\vert I_{\epsilon,1}\vert_{H_{\gamma}}=\lim_{\epsilon\rightarrow0}\left|\left[R_{\epsilon}^{\theta}(0,T)-\bar{R}^{\theta}(0,T)\right]u\right|_{H_{\gamma}}=0.
\end{equation}
For the second term, using Lemma \ref{lem: Semigroup property 1}
and Lemma \ref{lem: limit semigroup property 1}
\begin{align*}
\mathbb{E}\left|\int_{0}^{t}\left[S_{\epsilon}(t-s)-\bar{S}(t-s)^{\vee}\right]d\mathcal{W}(s)\right|_{H_{\gamma}} & \leq\mathbb{E}\left(\int_{0}^{t}\left|S_{\epsilon}(t-s)\right|_{L_{(HS)}(\mathscr{S}_{q}',H_{\gamma})}^{2}+\left|\bar{S}(t-s)^{\vee}\right|_{L_{(HS)}(\mathscr{S}_{q}',H_{\gamma})}^{2}ds\right)^{1/2}\\
 & \leq C_T\Vert m\Vert_{p}^{1/2}\left(\int_{0}^{t}(t-s)^{-(p-1)/p}ds\right)^{1/2}\\
 & =C_T\Vert m\Vert_{p}^{1/2}t^{1/2p},
\end{align*}
which is finite. Hence we can use the Burkholder-Davis-Gundy inequality  to obtain  
\begin{align*}
\mathbb{E}\,\left|I_{\epsilon,2}\right|_{H_{\gamma}}^{q} & =\mathbb{E}\,\left|\int_{0}^{T}\int_{s}^{T}\left[S_{\epsilon}(t-s)-\bar{S}(t-s)^{\vee}\right]\theta(t)dt\,d\mathcal{W}(s)\right|_{H_{\gamma}}^{q}\\
&=\mathbb{E}\,\left|\int_{0}^{T}\int_{0}^{T-s}\left[S_{\epsilon}(t)-\bar{S}(t)^{\vee}\right]\theta(t+s)dt\,d\mathcal{W}(s)\right|_{H_{\gamma}}^{q}\\
 & \leq C_{T,q}\left(\int_{0}^{T}\sum_{j=1}^{\infty}\left|\int_{0}^{T-s}\left[S_{\epsilon}(t)e_{j}-\bar{S}(t)^{\vee}e_{j}\right]\theta(t+s)dt\right|_{H_{\gamma}}^{2}ds\right)^{q/2}.
\end{align*}
By Lemma \ref{lem: Weak type convergence of SPDE} and the dominated
convergence theorem, we have that 
\[\lim_{\epsilon\rightarrow0}\mathbb{E}\,\left|I_{\epsilon,2}\right|_{H_{\gamma}}^{q}=0.\]
This, together with \eqref{cx31}, implies \eqref{cx30}.
\end{proof}


\end{document}